\documentclass[onefignum,onetabnum]{siamart190516}

\usepackage{lipsum}
\usepackage{amsfonts}
\usepackage{graphicx}
\usepackage{epstopdf}
\usepackage{algorithmic}
\ifpdf
  \DeclareGraphicsExtensions{.eps,.pdf,.png,.jpg}
\else
  \DeclareGraphicsExtensions{.eps}
\fi

\newsiamremark{remark}{Remark}
\newsiamremark{hypothesis}{Hypothesis}
\crefname{hypothesis}{Hypothesis}{Hypotheses}
\newsiamthm{claim}{Claim}

\headers{Contact problems in porous media}{L. Banz, F. Bertrand}

\title{Contact problems in porous media\thanks{Submitted to the editors \today.
}}

\author{L. Banz\thanks{Department of Mathematics, University of Salzburg, Salzburg, Austria 
  (\email{lothar.banz@plus.ac.at}).}
\and 
F. Bertrand\thanks{Faculty of Electrical Engineering, Mathematics and Computer Science, University of Twente, Enschede, Netherlands
  (\email{f.bertrand@utwente.nl}).}
}

\usepackage{amsopn}
\usepackage{amssymb}

\ifpdf
\hypersetup{
  pdftitle={Contact problems in porous media},
  pdfauthor={L. Banz, F. Bertrand}
}
\fi

\usepackage{stmaryrd}

\usepackage{subfig}
\usepackage{tikz,pgfplots}
\pgfplotsset{compat=newest}

\definecolor{darkbrown}{rgb}{0.57, 0.40, 0.13}

\newcommand{\R}{{\mathbb R}}
\renewcommand{\div}{\operatorname{div}}
\newcommand{\bV}{{\bf V}}
\newcommand{\bK}{{\bf K}}
\newcommand{\bG}{{\bf G}}

\newcommand{\bH}{{\bf H}}

\newcommand{\bI}{{\bf I}}
\newcommand{\bL}{{\bf L}}

\newcommand{\bu}{{\boldsymbol u}}
\newcommand{\bv}{{\boldsymbol v}}
\newcommand{\bw}{{\boldsymbol w}}
\newcommand{\bff}{{\boldsymbol f}}
\newcommand{\bfe}{{\boldsymbol f_e}}
\newcommand{\bl}{{\boldsymbol l}}
\newcommand{\bn}{{\boldsymbol n}}
\newcommand{\bt}{{\boldsymbol t}}
\newcommand{\Id}{\boldsymbol{I}}
\newcommand{\divu}{\operatorname{div} \bu}
\newcommand{\divv}{\operatorname{div} \bv}
\newcommand{\bsigma}{\boldsymbol \sigma}
\newcommand{\btheta}{\boldsymbol \theta}
\newcommand{\bkappa}{\boldsymbol \kappa}
\newcommand{\beps}{\boldsymbol{\varepsilon}}
\newcommand{\bepsu}{\beps(\bu)}

\newcommand*\norm[1]{\left|\hspace*{-0.035cm}\left|\hspace*{-0.035cm}\left|#1\right|\hspace*{-0.035cm}\right|\hspace*{-0.035cm}\right| }

\begin{document}

\maketitle

\begin{abstract}
The Biot problem of poroelasticity is extended by Signorini contact conditions. The resulting Biot contact problem is formulated and analyzed as a two field variational inequality problem of a perturbed saddle point structure. We present an a priori error analysis for a general as well as for a $hp$-FE discretization including convergence and guaranteed convergence rates for the latter. In particular, these rates are always optimal in the discretization parameters for the fluid pressure, and can be linear for the lowest order $h$-version. Moreover, we state a residual based a posteriori error estimator.
 Numerical results underline our theoretical findings and show that optimal, in particular exponential, convergence rates can be achieved by adaptive schemes for two dimensional problems. 
\end{abstract}



\section{Introduction}

Contact problems occur in a wide range of areas, including mechanical engineering or biomechanics. Because the colliding structures cannot penetrate each other, an inequality constraint of non-penetration is prescribed, leading to a mathematical description of the problem in form of a variational inequality. We refer to \cite{Duvaut1976,Glowinski1984, Hlavacek1988,Kikuchi1988Contact} and the references therein for an overview of the corresponding numerical analysis. 

In biomechanics however, the structure usually consists of porous material. Typical applications in biomedical engineering therefore includes the prediction of the stress between (porous) cartilage and a prosthesis as in \cite{guo2013biphasic}, between a porous bone and a cartilage as in \cite{Sahu201614} as well as the deformation of cartilage replacement material \cite{Yang2020}. To the best knowledge of the authors, the numerical analysis of the corresponding variational inequalities is restricted to elliptic bilinear forms and neglects the porosity of the material. 

For the description of a porous material, we refer to the seminal work of Biot \cite{Bio:41}. It describes the deformation of a fluid saturated elastic porous medium due to fluid flow, and the fluid flow subject to the deformation of the medium. This leads to a two-fields formulation, given by the conservation of mechanical momentum and the conservation of the fluid mass, expressed in the displacement field $\mathbf u$  and the fluid pressure $p$. It is shown in \cite{MurThoLou:96} that the
$H^1$-conforming Taylor–Hood finite element combination yields optimal a priori error estimates. Several three-fields formulations with various discretisations were proposed afterwards, to increase the robustness towards model parameters see e.g.~\cite{PhiWhe:08, OyaRui:16, LeeMarWin:17, Bertrand2020h, Bertrand_ExCo_2022}.
In the next section, we present our novel Biot contact problem and prove its well-posedness in Section~\ref{sec:well-posedness}. An abstract Galerkin approximation framework is introduced in Section~\ref{sec:abstractGalerkin}, leading to a highly potent $hp$-finite element discretization in Section~\ref{sec:hp_discretization} which also includes a convergence (rate) analysis. The theoretical results are then illustrated by a computational example
in Section~\ref{sec:numerics}.

\section{The Biot contact problem}
\label{sec:biot}
 The aim of this section is the derivation of equations describing the equilibrium position of a porous material with a fluid source $f_f$ subject to a loading force ${\bfe}$ and constrained to lie above a given rigid obstacle represented by a gap function $g$. To this purpose, the Biot equations are considered in a Lipschitz domain $\Omega \subset \R^d$ with $d=2,3$, i.e.~the problem consists of finding a vector-valued displacement field $\bu$ and a scalar-valued pressure field $p$ such that for the effective stress tensor $\btheta(\bu)$, the total stress tensor $\bsigma(\bu,p)$, and the fluid content $\phi(\bu,p)$ defined by
\begin{subequations}
\begin{align}
    \btheta(\bu)  &= 2 \tau \bepsu+\iota \divu \Id\\
    \bsigma(\bu,p)&= \btheta(\bu) - \alpha p \Id\\
    \phi(\bu,p)   &= \frac{\alpha^2}{\iota} p+\alpha \divu,
\end{align}
\end{subequations}
where $\bepsu = (\nabla \bu + \nabla^\top \bu)/2$ is the linearized strain tensor, $\iota>0$ and $\tau>0$ are the first and second Lamé-parameters (usually denoted by $\lambda$ and $\mu$ but these letters are needed elsewhere), $\alpha>0$ is the Biot–Willis constant and $\bkappa>0$ is a symmetric uniformly positive defined matrix describing the permeability, 
it holds
\begin{subequations}\label{biotsystem}
\begin{align}
 -\operatorname{div} \bsigma(\bu,p)  &=\bfe\\
\operatorname{div}\left(\bkappa \nabla p\right)-\phi(\bu,p)&=f_f  .
\end{align}
\end{subequations}
Since $\operatorname{div} \bsigma(\bu,p) = \operatorname{div} \btheta(\bu)-\alpha \nabla p$ the two-field system therefore reads
\begin{subequations} \label{biotsystem2F}
\begin{align}
-\div
\left(2\tau \beps(\bu)+\iota \operatorname{div} \bu \Id \right)
+\alpha \nabla p &=\bfe\\
\operatorname{div}\left(\bkappa \nabla p\right)
-\frac{\alpha^{2}}{\iota} p-\alpha \operatorname{div} \bu
&=f_f  .
\end{align}
\end{subequations}
\begin{figure}[t]
  \centering
  \includegraphics[scale=1.1]{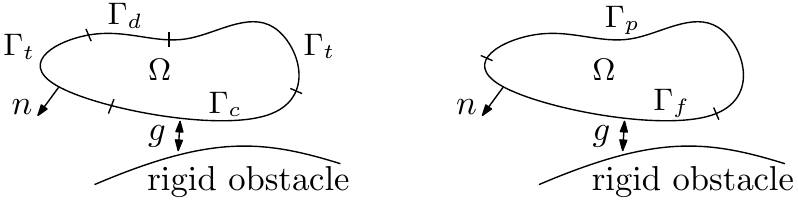}
  \caption{The domain under consideration for the Biot contact problem}
  \label{fig:biotobstaclemodel}
\end{figure}

For further scaling of these equations we refer to \cite{LeeMarWin:17}. To complete the system \eqref{biotsystem2F}, suitable boundary conditions have to be prescribed. To this purpose, the boundary is decomposed twice into non-overlapping, disjoint subsets that cover the entire boundary, namely $\Gamma = \partial \Omega=\Gamma_{p} \cup \Gamma_{f}=\Gamma_{c} \cup \Gamma_{d} \cup \Gamma_{t}.$ 
We assume that both $\Gamma_{p}$ and $\Gamma_{d}$ have positive measure, $\overline{\Gamma_d}\cap \overline{\Gamma_c} = \emptyset$ {for simplicity}, and prescribe the homogeneous mixed boundary conditions
 \begin{alignat*}{2}
 p&=0 \text{ on }\Gamma_{p}, 
 &\qquad 
 \bkappa \nabla {p} \cdot \mathbf{n}&= 0\text{ on }\Gamma_{f}
 \\
  \mathbf{u}&=\mathbf{0}\text{ on }\Gamma_{d},
  &\qquad 
  \bsigma \cdot \mathbf{n}&=\mathbf{0}\text{ on }\Gamma_{t} 
 \end{alignat*}
 where $\mathbf{n}$ is the outward unit normal vector. The incorporation of inhomogeneous Neumann data is straightforward and only lengthens the formulas and {proofs}.
 On the remaining boundary part $\Gamma_c$, which potentially comes into a frictionless contact with the rigid obstacle, we apply the so-called Signorini contact conditions, see e.g.~\cite{Kikuchi1988Contact}, 
 \begin{align}
  \label{eq:Signori}
   \bu_\bn \leq g , \qquad \bsigma_\bn \leq 0, \qquad  \bsigma_\bn \left[
     \bu_\bn - g \right]= 0
     \qquad \text{on }\Gamma_c 
 \end{align}
 with normal displacement $\bu_\bn = \bu \cdot \bn$ and contact pressure $\bsigma_\bn = (\bsigma \cdot \bn) \cdot \bn$,
 and the no friction condition
 $    \bsigma_\bt = \mathbf{0} \text{ on }\Gamma_c$
 where the tangential stress is given by $\bsigma_\bt = \bsigma \cdot \bn - \bsigma_\bn \bn$.
 
 These boundary conditions naturally lead to the Hilbert spaces
$$ V = H_{0, \Gamma_{p}}^{1}(\Omega)=\left\{ v \in H^{1}(\Omega):\left.v\right|_{\Gamma_{p}}=0\right\} \quad
\text{and} \quad \bV = \left[H_{0, \Gamma_{d}}^{1}(\Omega)\right]^{d}$$
and {to the} following variational formulation
\begin{subequations} \label{biotsystemVF1}
\begin{align}
-\langle\div
\left(2\tau \beps(\bu)+\iota \operatorname{div} \bu \Id\right),\bv
\rangle+\left(\alpha \nabla p,\bv\right)&=\langle\bfe,\bv \rangle\\
\langle \operatorname{div}\left(\bkappa \nabla p\right)
-\frac{\alpha^{2}}{\iota} p-\alpha \operatorname{div} \bu,q\rangle
&=\langle f_f,q \rangle 
\end{align}
\end{subequations}
where $\langle \cdot,\cdot \rangle$ is the duality pairing between the involved Hilbert spaces naturally extending the $L^2$-inner product $(\cdot,\cdot)$. An index at these indicates the integration domain if it is not $\Omega$.
Integration by parts leads to
\begin{align*}
 \label{biotsystemVF2}
\begin{split}
(2\tau \beps(\bu)\!+\!\iota (\divu) \Id,\nabla \bv)-\left(\alpha p,\divv \right)
- \langle \left(2\mu \beps(\bu)\!+\!\iota \operatorname{div} \bu \Id \!-\! \alpha p \Id\right) \cdot \bn,
\bv
\rangle_{\Gamma}
&=\langle \bfe,\bv \rangle\\
-
\left(\bkappa \nabla p, \nabla q\right)
+\langle
\bkappa \nabla p \cdot \bn, q
\rangle_\Gamma
-(\frac{\alpha^{2}}{\iota} p+\alpha \operatorname{div} \bu
,q)
&=\langle f_f,q \rangle  .
\end{split}
\end{align*}
With the use of the boundary conditions, any $(\bv,q)\in \bV\times V$ allows for
\begin{subequations} \label{biotsystemVF3}
\begin{align} 
2\tau (\beps(\bu), \beps(\bv))+(\iota\divu-\alpha p, \divv)
&=\langle \bfe, \bv \rangle +
\langle
\bsigma \cdot \bn, \bv \rangle_{\Gamma_c} \label{biotsystemVF3A}
\\
-\alpha\left(\operatorname{div} \bu, q\right)-\frac{\alpha^{2}}{\iota}\left(p, q\right)-\left(\bkappa \nabla p, \nabla q\right)&=\langle f_f, q\rangle \label{biotsystemVF3B}.
\end{align}
\end{subequations}
If we restrict the test functions $\bv$ to lie in
\begin{align}
    \bK = \{ \bv \in \bV : \ \bv_\bn \leq g \text{ on } \Gamma_c\} 
\end{align}
we can eliminate the unknown quantity $\langle
\bsigma \cdot \bn, \bv \rangle_{\Gamma_c}$ on the right hand side.
Obviously, $\bK$ is a closed and convex subset of $\bV$. If further $g \in H^{1/2}(\Gamma_c)$ is sufficiently large, e.g.~$g\geq 0$, then $\bK \neq \emptyset$, which we assume from now on.
Taking the test function in \eqref{biotsystemVF3A} to be $\bv-\bu$ with some $\bv \in \bK$ we obtain with \eqref{eq:Signori} and the no-friction condition
\begin{align*}
\langle \bsigma\cdot\bn, \bv - \bu \rangle_{\Gamma_c}
&= \langle \bsigma_\bn,  \bv_\bn - \bu_\bn  \rangle_{\Gamma_c} + \langle \bsigma_\bt, \bv - \bu\rangle_{\Gamma_c} \\
& = \langle \bsigma_\bn, \bv_\bn-g \rangle_{\Gamma_c} \\
& \geq 0 .
\end{align*}
Therewith, the Biot contact problem can be written as a variational inequality problem with a perturbed saddle point problem structure. That is: Find a pair $(\bu,p) \in \bK \times V$ such that
\begin{subequations}\label{VIabstract}
\begin{alignat}{2}
a(\bu,\bv - \bu) - b(\bv - \bu,p) &\geq \langle \bfe, \bv - \bu \rangle & \quad& \forall\ \bv \in \bK \label{VIabstractA}\\
-b(\bu,q)-c(p,q)&=\langle f_f, q \rangle &\quad& \forall\ q \in V  \label{VIabstractB}
\end{alignat}
\end{subequations}
with the three bilinear forms
\begin{align*}
a(\bu,\bv) &= 2\tau(\beps(\bu), \beps(\bv))+\iota(\operatorname{div} \bu, \operatorname{div} \bv) & \quad \\
 b(\bu,q) &= \alpha\left(\operatorname{div} \bu,q\right)\\
 c(p,q) &= \frac{\alpha^{2}}{\iota}\left(p, q\right)+\left(\bkappa \nabla p, \nabla q\right) .
\end{align*}
 We emphasize that for $|\Gamma_c|=0$ we have $\bK = \bV$ and the system \eqref{VIabstract} reduces to the classical two field formulation of the Biot problem \cite{MurThoLou:96}.
Note that $a(\cdot,\cdot)$ and $c(\cdot,\cdot)$ are symmetric and even inner products inducing the parameter-dependent norms
\begin{alignat*}{2}
\norm{\bv}_{\bV}^{2}
&= a(\bv,\bv) =2\tau \|\beps(\bv)\|_{0}^{2}+\iota\|\operatorname{div} \bv\|_{0}^{2}  &\quad & \forall \ \bv \in \bV \\
\norm{q}_V^{2}
&=c(q,q)=\frac{\alpha^{2}}{\iota}\left\|q\right\|_{0}^{2}+\left(\bkappa \nabla q, \nabla q\right)  & \quad & \forall \ q \in V. 
\end{alignat*}
Due to Korn's inequality these {norms} are equivalent to the standard  $[H^1_{0,\Gamma_d}(\Omega)]^d$ and $H^1_{0,\Gamma_p}(\Omega)$ norms
but allow for parameter robust estimations. For the latter we also need the dual norms
\begin{align}
    \norm{ \cdot }_{\bV^*} = \sup_{\bv \in \bV \setminus\{\mathbf{0}\} } \frac{\langle \cdot , \bv \rangle}{\norm{\bv}_{\bV}} \quad \text{and} \quad
    \norm{ \cdot }_{V^*} = \sup_{q \in V \setminus\{0\} } \frac{\langle \cdot , q \rangle}{\norm{q}_V},
\end{align}
which are equivalent to the standard (dual) norms of $([H^1_{0,\Gamma_d}(\Omega)]^d)^*$ and $(H^1_{0,\Gamma_p}(\Omega))^*$ as well.
Let $\gamma: \bV \rightarrow H^{1/2}(\Gamma_c)$ such that $\gamma \bv = \bv|_{\Gamma_c}\cdot \bn$ is the normal component of the trace restricted to $\Gamma_c$. Moreover, we need to equip the trace space $H^{1/2}(\Gamma_c)$ with the norm
\begin{align} \label{eq:defTraceNorm}
    \norm{v}_{\bV,\Gamma_c} = \inf_{\bw \in \bV,\;\gamma \bw = v} \norm{\bw}_{\bV}
\end{align}
which is equivalent to various standard $\| \cdot \|_{H^{1/2}(\Gamma_c)}$ norms, and denote by
\begin{align} \label{eq:defDualTraceNorm}
    \norm{\mu}_{\bV^*,\Gamma_c} = \sup_{ v \in H^{1/2}(\Gamma_c)\setminus\{0\}} \frac{\langle \mu,v\rangle_{\Gamma_c}}{\norm{v}_{\bV,\Gamma_c}}
\end{align}
its dual norm which is equivalent to the $\tilde{H}^{-1/2}(\Gamma_c)$ dual norm.
We summarize the equivalence constants in the following lemma.
\begin{lemma} \label{lem:norm_equi}
There holds
\begin{align*}
    2\tau c_{Korn} \|\bv \|_{H^1(\Omega)}^2 &\leq \norm{\bv}_{\bV}^2 \leq (2\tau+d \iota)\|\bv \|_{H^1(\Omega)}^2 \\
    \operatorname{ev}_{\min}(\bkappa)c_{Poin} \|q \|_{H^1(\Omega)}^2 &\leq \norm{q}_V^2 \leq \max\left\{ \frac{\alpha^2}{\iota},\operatorname{ev}_{\max}(\bkappa) \right\}\|q \|_{H^1(\Omega)}^2 \\
        (2\tau+d \iota)^{-1}\|\bff \|_{([H^1_{0,\Gamma_d}(\Omega)]^d)^*}^2 &\leq \norm{\bff}_{\bV^*}^2 \leq (2\tau c_{Korn})^{-1} \|\bff \|_{([H^1_{0,\Gamma_d}(\Omega)]^d)^*}^2 \\
    \left(\max\left\{ \frac{\alpha^2}{\iota},\operatorname{ev}_{\max}(\bkappa) \right\}\right)^{-1}\!\!\! \|f \|_{(H^1_{0,\Gamma_p}(\Omega))^*}^2 &\leq \norm{f}_{V^*}^2 \leq (\operatorname{ev}_{\min}(\bkappa)c_{Poin})^{-1} \|f \|_{(H^1_{0,\Gamma_p}(\Omega))^*}^2 \\
     2\tau c_{Korn} \|v \|_{H^{1/2}(\Gamma_c)}^2 &\leq \norm{v}_{\bV,\Gamma_c}^2 \leq (2\tau+d \iota)\|v \|_{H^{1/2}(\Gamma_c)}^2 \\
     (2\tau+d \iota)^{-1}\|\mu \|_{\tilde{H}^{-1/2}(\Gamma_c)}^2 &\leq \norm{\mu}_{\bV^*,\Gamma_c}^2 \leq (2\tau c_{Korn})^{-1}\|\mu \|_{\tilde{H}^{-1/2}(\Gamma_c)}^2
\end{align*}
where $c_{Korn}$, $c_{Poin}$ is the {ellipticity} constant from Korn's, Poincar\'e's inequality, respectively, $\operatorname{ev}_{\min}(\bkappa)$, $\operatorname{ev}_{\max}(\bkappa)$ is the minimal, maximal eigenvalue of $\bkappa$, respectively, and the standard trace norm is defined by
\begin{align*}
    \|v\|_{H^{1/2}(\Gamma_c)} = \inf_{\bw \in \bV,\;\gamma \bw = v} \|\bw\|_{H^1(\Omega)}.
\end{align*}
\end{lemma}

\section{Well-posedness}
\label{sec:well-posedness}
Throughout the rest of the paper let $\bfe \in \bV^*$ and $f_f \in V^*$. 
Since $a(\cdot,\cdot)$ and $c(\cdot,\cdot)$ are inner products inducing the norms $\norm{\cdot}_{\bV}$ and $\norm{\cdot}_V$, respectively, their {ellipticity} and continuity constants are one. Moreover, the Cauchy-Schwarz inequality implies the continuity of $b(\cdot,\cdot)$ with respect to these norms, i.e.
\begin{align} \label{B_continuous}
    b(\bv,q) \leq \left( \iota\|\operatorname{div} \bv\|_{0}^{2} \right)^{1/2}\left( \frac{\alpha^{2}}{\iota}\left\|q\right\|_{0}^{2} \right)^{1/2} \leq \norm{\bv}_{\bV} \norm{q}_V.
\end{align}

For the analysis but not for the numerics it is helpful to eliminate the pressure field $p$ from \eqref{VIabstract}. For that let $A:\bV \rightarrow \bV^*$, $B:\bV \rightarrow V^*$ and $C:V \rightarrow V^*$ be the operators associated with the three bilinear forms, i.e.
\begin{alignat*}{2}
    \langle A \bu,\bv\rangle &= a(\bu,\bv) &\quad& \forall\ \bu,\bv \in \bV \\
    \langle B \bv,q\rangle &= b(\bv,q) &\quad& \forall\ \bv \in \bV, \ \forall\ q \in V \\
    \langle C p,q\rangle &= c(p,q) &\quad& \forall\ p,q \in V.
\end{alignat*}
We denote by $B^\top:V \rightarrow \bV^*$ the adjoint operator of $B$. Recall that these operators inherit all the properties form their respective bilinear {forms}.
With $C$ at hand we may solve \eqref{VIabstractB} for
\begin{align} \label{SolutionP}
    p = -C^{-1} \left(B \bu + f_f \right)
\end{align}
and, therewith, reduce \eqref{VIabstract} to the variational inequality: Find $\bu \in \bK$ such that
\begin{align} \label{ReducedVI}
  \langle D \bu, \bv-\bu \rangle \geq \langle \bl ,\bv-\bu\rangle \quad \forall \ \bv \in \bK
\end{align}
with $D=A+B^\top C^{-1} B : \bV \rightarrow \bV^*$ and $\bl = \bfe - B^\top C^{-1} f_f \in \bV^*$.
Since $A$ and $C$ are symmetric, continuous and {elliptic}, and $B$ is continuous, we find that the new operator $D$ is symmetric, continuous and $\bV$-{elliptic} as well.

\begin{lemma} \label{lem:D_coercive}
The Operator $C^{-1}$ is continuous with continuity constant one.
The Operator $D$ is symmetric, continuous and {elliptic} with continuity constant two and {ellipticity} constant one.
\end{lemma}
\begin{proof}
 Let $f \in V^*$ be arbitrary. Then
 \begin{align*}
     \norm{C^{-1} f}_{V}^2 = \langle C (C^{-1} f), C^{-1} f \rangle = \langle f, C^{-1} f \rangle \leq \norm{f}_{V^*} \norm{C^{-1} f}_{V},
 \end{align*}
 which shows that the continuity constant of $C^{-1}$ is one.
 
 Let $\bu$, $\bv \in \bV$ be arbitrary. Then
 \begin{align*}
     \langle D \bu,\bv \rangle &= \langle A \bu,\bv \rangle + \langle B^\top C^{-1} B \bu,\bv \rangle \\
     &= \langle A \bu,\bv \rangle + \langle  C^{-1} B \bu,B \bv \rangle \\
     &\leq \norm{\bu}_{\bV}\norm{\bv}_{\bV} + \norm{B\bu}_{V^*}\norm{B\bv}_{V^*} \\
     & \leq 2 \norm{\bu}_{\bV}\norm{\bv}_{\bV}.
 \end{align*}
 Thus, the continuity constant of $D$ is two.
 The {ellipticity} follows from
  \begin{align*}
     \langle D \bv,\bv \rangle = \langle A \bv,\bv \rangle + \langle C^{-1} B \bv,B\bv \rangle = \norm{\bv}_{\bV}^2 + \norm{C^{-1} B \bv}_{V}^2 \geq \norm{\bv}_{\bV}^2
 \end{align*}
 for all $\bv \in \bV.$ The symmetry of $D$ follows {directly} from the symmetry of $A$ and $C$.
\end{proof}

The previous lemma guarantees that the reduced problem \eqref{ReducedVI} falls within the scope of elliptic variational inequalities of the first kind while the original problem \eqref{VIabstract} does not.

\begin{remark}
By subtracting \eqref{VIabstractB} from \eqref{VIabstractA} we obtain the equivalent variational inequality: Find $(\bu,p) \in \bK \times V$ such that
\begin{align} \label{eq:directVI}
\tilde{a}((\bu,p),(\bv - \bu,q-p)) \geq \langle \bfe, \bv - \bu \rangle - \langle f_f,q-p \rangle  \quad \forall\ (\bv,q) \in \bK\times V
\end{align}
with 
\begin{align*}
    \tilde{a}((\bu,p),(\bv,q)) = a(\bu,\bv) - b(\bv,p) + b(\bu,q) + c(p,q). 
\end{align*}
Indeed $\tilde{a}(\cdot,\cdot)$ is continuous and {elliptic} and thus the Biot contact problem can be analyzed as an elliptic variational inequality of the first kind. As
\begin{gather*}
    \tilde{a}((\bv,q),(\bv,q)) = a(\bv,\bv) + c(q,q) = \norm{\bv}_{\bV}^2 + \norm{q}_{V}^2, \\
    \tilde{a}((\bu,-C^{-1} B \bu),(\bv,q)) = \langle A \bu,\bv \rangle + \langle B \bv, C^{-1} B \bu \rangle = \langle D \bu,\bv \rangle \leq 2 \norm{\bu}_{\bV}\norm{\bv}_{\bV}
\end{gather*}
the {ellipticity} and continuity constants are not better than for $D$. Hence, we cannot necessarily expect better constants in the subsequent estimates but only shorter and simpler proofs. A direct or naive application of known results for elliptic variational inequalities to \eqref{eq:directVI} would lead to reduced guaranteed convergence rates in both $\bu$ and $p$ while our approach yields reduced convergence rate in $\bu$ only and optimal rates in $p$, see Theorem~\ref{thm:convergence_rates}.
Moreover, as $\tilde{a}(\cdot,\cdot)$ is not symmetric, the problem \eqref{eq:directVI} is not (directly) equivalent to a minimization problem unlike \eqref{ReducedVI}. If one wants to extend the Biot contact problem by (Tresca) friction, the two problems \eqref{eq:directVI} and \eqref{ReducedVI} become variational inequalities of the second kind, but the existence of a solution is only immediately guaranteed for \eqref{ReducedVI} via an equivalent minimization problem. Hence, we pursue \eqref{ReducedVI} to obtain better guaranteed convergence rate and to lay the foundation for a later generalization of the Biot contact problem by friction.
\end{remark}

\begin{theorem} \label{thm:existence}
  There exists exactly one solution $(\bu,p)$ to \eqref{VIabstract}. Moreover, that solution depends Lipschitz-continuously on the data $(\bfe,f_f,g)$, i.e.
\begin{align*}
\left(\norm{\bu-\widetilde{\bu}}_{\bV}^2 + \norm{ p - \widetilde{p} }_V^2\right)^{1/2}
     &\leq  \norm{\bfe- \widetilde{\bfe}}_{\bV^*}+\norm{f_f - \widetilde{f_f} }_{V^*}  +3\norm{g-\widetilde{g}}_{\bV,\Gamma_c}
  \end{align*}
  for $(\bu,p)$, $(\widetilde{\bu},\widetilde{p})$ the solution to the data $(\bfe,f_f,g)$, $(\widetilde{\bfe},\widetilde{f_f},\widetilde{g})$, respectively.
\end{theorem}
\begin{proof}
  Since $D$ is continuous and $\bV$-{elliptic} (see Lemma~\ref{lem:D_coercive}), $\bl \in \bV^*$ and $\bK$ is a non-empty, closed, convex set, the Stampacchia theorem implies unique existence of a $\bu \in \bK$ solving \eqref{ReducedVI}. For a given $\bu \in \bV$ the Lax-Milgram lemma implies the unique solvability of \eqref{VIabstractB} with the solution $p$ given by \eqref{SolutionP}. Due to the equivalence of \eqref{SolutionP}-\eqref{ReducedVI} with \eqref{VIabstract} the pair $(\bu,p)$ is the unique solution of \eqref{VIabstract}.
  
   Subtracting the two to the data $(\bfe,f_f,g)$ and  $(\widetilde{\bfe},\widetilde{f_f},\widetilde{g})$ corresponding \eqref{VIabstractB} from each other gives
  \begin{align} \label{DiffVIabstractB}
      -b(\bu - \widetilde{\bu},q)-c(p-\widetilde{p},q)=\langle f_f - \tilde{f}_f,q \rangle \quad \forall \ q \in V.
  \end{align}
  Next, let $\bG \in \bV$ be some lifting of $g$, i.e.~$\gamma(\bG) = g$, which existence is guaranteed by the trace theorem. Likewise $\widetilde{\bG}\in \bV$ is some lifting of $\widetilde{g}$. Furthermore, let $\bK_0 = \{ \bv \in \bV: \bv_\bn \leq 0 \text{ on } \Gamma_c\}$, i.e.~$ \bu =  \bG + \bu_0   \in \bG + \bK_0$ and $\widetilde{\bu} =\widetilde{\bG} +  \widetilde{\bu}_0  \in \widetilde{\bG} + \bK_0$.
 Therewith, the two \eqref{VIabstractA} become
 \begin{align*}
     a(\bu_0 + \bG, \bv - \bu_0) - b( \bv - \bu_0,p) &\geq \langle \bfe ,  \bv - \bu_0\rangle \quad \forall \ \bv \in \bK_0\\
     a(\widetilde{\bu}_0 + \widetilde{\bG}, \widetilde{\bv} - \widetilde{\bu}_0) - b( \widetilde{\bv} - \widetilde{\bu}_0,p) &\geq \langle \widetilde{\bfe} ,  \widetilde{\bv} - \widetilde{\bu}_0\rangle \quad \forall \ \widetilde{\bv} \in \bK_0.
 \end{align*}
  Choosing $\bv = \widetilde{\bu}_0$ and $\widetilde{\bv} = \bu_0$, and afterwards adding the resulting inequalities yields
    \begin{align*}
      \langle \widetilde{\bfe}-\bfe ,  \widetilde{\bu}_0-\bu_0\rangle + a(\bG-\widetilde{\bG} , \widetilde{\bu}_0-\bu_0) & \geq a(\bu_0-\widetilde{\bu}_0,\bu_0-\widetilde{\bu}_0) + b(\widetilde{\bu}_0-\bu_0 ,p-\widetilde{p}) \\
      & = a(\bu_0-\widetilde{\bu}_0,\bu_0-\widetilde{\bu}_0) + c(p-\widetilde{p},p-\widetilde{p}) \\
      & \quad - b(\widetilde{\bG}-\bG,p-\widetilde{p}) + \langle f_f - \widetilde{f_f},p-\widetilde{p} \rangle \\
      & = \norm{ \bu_0-\widetilde{\bu}_0 }_{\bV}^2 + \norm{ p - \widetilde{p} }_V^2\\
      & \quad - b(\widetilde{\bG}-\bG,p-\widetilde{p}) + \langle f_f - \widetilde{f_f},p-\widetilde{p} \rangle 
  \end{align*}
  where we employ \eqref{DiffVIabstractB} in the second line. Duality arguments, the continuity of $a(\cdot,\cdot)$ and $b(\cdot,\cdot)$ now imply 
  \begin{align*}
      \norm{ \bu_0-\widetilde{\bu}_0 }_{\bV}^2 +  \norm{ p - \widetilde{p} }_V^2 &\leq \left(\norm{ \bfe- \widetilde{\bfe}}_{\bV^*}  +\norm{ \bG-\widetilde{\bG} }_{\bV}\right) \norm{ \bu_0-\widetilde{\bu}_0 }_{\bV} \\
     & \quad +\left(\norm{ \bG-\widetilde{\bG} }_{\bV} +\norm{ f_f - \widetilde{f_f} }_{V^*} \right) \norm{ p - \widetilde{p} }_V .
  \end{align*}
  Choosing the lifting $\bG$ and $\widetilde{\bG}$ to minimize $\norm{ \bG-\widetilde{\bG} }_{\bV}$ we obtain with the definition of the trace norm $\norm{\cdot}_{\bV,\Gamma_c}$ and trivial algebra that
    \begin{align*}
      \left(\norm{ \bu_0-\widetilde{\bu}_0 }_{\bV}^2 + \norm{ p - \widetilde{p} }_V^2\right)^{1/2} \leq  \norm{ \bfe- \widetilde{\bfe} }_{\bV^*} + \norm{ f_f - \widetilde{f_f} }_{V^*} + 2\norm{ g-\widetilde{g} }_{\bV,\Gamma_c} .
  \end{align*}
  Finally, the assertion follows with the triangle inequality.
\end{proof}

\begin{corollary} \label{cor:aprioriEst}
The solution $(\bu,p)$ of \eqref{VIabstract} satisfies the a priori estimate
\begin{align*}
    \left(\norm{ \bu }_{\bV}^2 + \norm{ p }_V^2\right)^{1/2} &\leq  \norm{ \bfe }_{\bV^*} + \norm{ f_f }_{V^*} +3\norm{g}_{\bV,\Gamma_c}.
\end{align*}
If $\mathbf{0}\in \bK$, then there holds
\begin{align*}
    \left(\norm{ \bu }_{\bV}^2 + \norm{ p }_V^2\right)^{1/2} &\leq  \norm{ \bfe }_{\bV^*} + \norm{ f_f }_{V^*}.
\end{align*}
\end{corollary}

\begin{proof}
The first assertion follows directly from the Lipschitz-continuous dependency of the solution on the data, see Theorem~\ref{thm:existence}, as $(\bu,p)=(\mathbf{0},0)$ is the solution of \eqref{VIabstract} to the data $(\bfe,f_f,g) = (\mathbf{0},0,0)$.

If $\mathbf{0}\in \bK$, then choosing $\bv =\mathbf{0}$ and $q=p$ in \eqref{VIabstract} yields
\begin{align*}
    \langle \bfe,  \bu\rangle \geq a(\bu, \bu) - b(\bu,p) = a(\bu, \bu) + c(p,p) + \langle f_f, p \rangle = \norm{\bu}_{\bV}^2 + \norm{p}_V^2 +  \langle f_f, p \rangle.
\end{align*}
Duality arguments yield the second assertion.
\end{proof}

{For the a priori error analysis it is helpful to formally link the contact pressure $\bsigma_\bn$ with the residual of some variational inequality. Indeed, the definition of $D$, \eqref{biotsystemVF3A} and the no-friction condition reveal that 
\begin{align} \label{eq:contactPressure2Residuum}
    \langle D \bu - \bl, \bv \rangle = \langle \bsigma_\bn(\bu,p) , \bv_\bn \rangle_{\Gamma_c} \quad \forall \ \bv \in \bV.
\end{align}
Moreover, 
\begin{align*}
  -\bsigma_\bn(\bu,p) \in  \Lambda = \left\{ \mu \in \tilde{H}^{-1/2}(\Gamma_c) : \langle \mu,v \rangle \geq 0 \ \forall \ v \in H^{1/2}_+(\Gamma_c) \right\}
\end{align*}
where
\begin{align*}
H^{1/2}_+(\Gamma_c) &= \left\{v \in H^{1/2}(\Gamma_c):  v \geq 0 \right\}.
\end{align*}
}

\section{Generic Galerkin approximation}
\label{sec:abstractGalerkin}
Let $\bV_N \subset \bV$ and $V_M \subset V$ be two finite dimensional subsets. Furthermore, let $\bK_N \subset \bV_N$ be a non-empty, closed, convex set not necessarily a subset of $\bK$. The non-conformity of $\bK_N$ comes from a possible discretization of the non-penetration constraint $\gamma(\bv) \leq g$. 

The discrete problem is to find a pair $(\bu_N,p_M)\in \bK_N \times V_M$ such that 
\begin{subequations}\label{DiscreteVIabstract}
\begin{alignat}{2}
a(\bu_N,\bv_N - \bu_N) 
-b(\bv_N - \bu_N,p_M)
&\geq \langle \bfe, \bv_N - \bu_N \rangle
&
\quad& \forall\  \bv_N \in \bK_N \label{DiscreteVIabstractA}\\
-b(\bu_N,q_M)-c(p_M,q_M)&=\langle f_f, q_M \rangle &\quad& \forall\ q_M \in V_M  . \label{DiscreteVIabstractB}
\end{alignat}
\end{subequations}

As on the continuous level we can reduce \eqref{DiscreteVIabstract} to a variational inequality in $\bu_N$ only but the necessary steps are a bit more technical. 
The Lax-Milgram lemma implies the existence of a linear mapping $R_M:V^* \rightarrow CV_M$ such that
\begin{align}
\langle R_M f, q_M \rangle = \langle f,q_M\rangle  \quad \forall \ q_M \in V_M,\ \forall \ f \in V^*.
\end{align}
Indeed, $R_M$ projects the right hand side data onto such a set that the Galerkin solution $p_M$ of \eqref{DiscreteVIabstractB} becomes
    $p_M = -C_M^{-1} \left(B \bu_N + f_f \right)$
with $C_M^{-1}=C^{-1}R_M: V^* \rightarrow V_M$ an approximation of $C^{-1}$. Inserting 
$p_M$ into \eqref{DiscreteVIabstractA} yields the reduced variational inequality: Find $\bu_N \in \bK_N$ such that
\begin{align} \label{DiscreteReducedVI}
  \langle D_M \bu_N, \bv_N-\bu_N \rangle \geq \langle \bl_M ,\bv_N-\bu_N\rangle \quad \forall \ \bv_N \in \bK_N
\end{align}
with $D_M=A+B^\top C^{-1}_M B: \bV \rightarrow \bV^*$ and $\bl_M = \bfe - B^\top C^{-1}_M f_f \in \bV^*$. Compared to \eqref{ReducedVI} not only the set $\bK$ is discretized but also the operator $D$ and the right hand side $\bl$ themselves.

\begin{lemma} \label{lem:DM_coercive}
The discrete Operator $C^{-1}_M$ is $M$-uniformly continuous with continuity constant one.
The discrete Operator $D_M$ is symmetric, $M$-uniformly continuous and $M$-uniformly {elliptic} with continuity constant two and {ellipticity} constant one.
\end{lemma}

\begin{proof}
Let $f \in V^*$ be arbitrary. Then
 \begin{align*}
     \norm{C^{-1}_M f}_{V}^2 = \langle C (C^{-1}_M f), C^{-1}_M f \rangle = \langle R_M f, C^{-1}_M f \rangle = \langle f, C^{-1}_M f \rangle \leq \norm{f}_{V^*} \norm{C^{-1}_M f}_{V}
 \end{align*}
 as $C^{-1}_M f \in V_M$. Hence, the continuity constant of $C^{-1}_M$ is one.
 
 Let $\bu$, $\bv \in \bV$ be arbitrary. Then there holds
 \begin{align*}
     \langle D_M \bu,\bv \rangle &= \langle A \bu,\bv \rangle + \langle B^\top C^{-1}_M B \bu,\bv \rangle \\
    & = \langle A \bu,\bv \rangle + \langle  C^{-1}_M B \bu,B \bv \rangle \\
     &\leq \norm{\bu}_{\bV}\norm{\bv}_{\bV} + \norm{B\bu}_{V^*}\norm{B\bv}_{V^*} \\
     & \leq 2 \norm{\bu}_{\bV}\norm{\bv}_{\bV}
 \end{align*}
 which implies a continuity constant of two for $D_M$.
 
 Let $f \in V^*$ be still arbitrary. Then
 \begin{align*}
     \langle C^{-1}_M f,f \rangle =  \langle C^{-1}_M f,R_M f \rangle = \langle C^{-1} R_M f, R_M f \rangle = \norm{C^{-1}R_M f}_{V}^2 \geq 0
 \end{align*}
 as $C^{-1}_M f \in V_M$. Therewith, we find that
 \begin{align*}
     \langle D_M \bv,\bv \rangle = \langle A \bv,\bv \rangle + \langle C^{-1}_M B \bv,B\bv \rangle \geq \langle A \bv,\bv \rangle = \norm{\bv}_{\bV}^2
 \end{align*}
 for all $\bv \in \bV.$
 For the symmetry of $D_M$ we observe that $A$ and $C^{-1}$ are symmetric and $C^{-1}_M=C^{-1}R_M$ maps onto $V_M$. Thus we obtain with the definition of $R_M$ that
 \begin{align*}
     \langle D_M \bu,\bv\rangle & = \langle A\bu,\bv \rangle + \langle C^{-1}_M B \bu, B\bv \rangle \\
     &= \langle A\bu,\bv \rangle + \langle C^{-1}_M B \bu, R_M B\bv \rangle \\
     &= \langle A\bu,\bv \rangle + \langle R_M B \bu, C^{-1}R_M B\bv \rangle \\
     &= \langle A\bu,\bv \rangle + \langle R_M B \bu, C^{-1}_M B\bv \rangle\\
     &= \langle \bu,A\bv \rangle + \langle  B \bu, C^{-1}_M B\bv \rangle \\
     & =  \langle \bu,D_M \bv\rangle
 \end{align*}
 for all $\bu$, $\bv \in \bV$.
\end{proof}

\begin{theorem} \label{thm:DiscreteExistence}
  There exists exactly one solution $(\bu_N,p_M)$ to \eqref{DiscreteVIabstract}. Moreover, that solution satisfies the stability estimate
\begin{align*}
\left(\norm{\bu_N}_{\bV}^2 + \norm{p_M}_{\bV}^2 \right)^{1/2}
     &\leq 2\norm{\bfe}_{\bV^*}+\norm{f_f}_{V^*}  + 3 \inf_{\bv_N \in \bK_N} \norm{\bv_N}_{\bV}
  \end{align*}
  and if $\mathbf{0}\in \bK_N$ then even 
  \begin{align*}
\left(\norm{\bu_N}_{\bV}^2 + \norm{p_M}_{\bV}^2 \right)^{1/2}
     &\leq \norm{\bfe}_{\bV^*}+\norm{f_f}_{V^*}  .
  \end{align*}
\end{theorem}
\begin{proof}
  Due to Lemma~\ref{lem:DM_coercive} the existence and uniqueness proof as well as the second stability estimate follow analogously to Theorem~\ref{thm:existence} and Corollary~\ref{cor:aprioriEst}.
  For the first stability estimate we find with \eqref{DiscreteVIabstract} and duality arguments that
  \begin{align*}
      \norm{\bu_N}_{\bV}^2 + \norm{p_M}_{\bV}^2 &= a(\bu_N,\bu_N) + c(p_M,p_M) \\
      &= a(\bu_N,\bu_N) - b(\bu_N,p_M) - \langle f_f,p_M \rangle \\
      &\leq \langle \bfe, \bu_N-\bv_N \rangle +a(\bu_N,\bv_N)-b(\bv_N,p_M)- \langle f_f,p_M \rangle \\
      & \leq \norm{\bu_N}_{\bV} \left( \norm{\bfe}_{\bV^*} +\norm{\bv_N}_{\bV} \right) +\norm{\bfe}_{\bV^*}\norm{\bv_N}_{\bV} \\
      & \quad + \norm{p_M}_{\bV} \left( \norm{f_f}_{\bV^*} + \norm{\bv_N}_{\bV} \right).
  \end{align*}
  Young's inequality and the subadditivity of the square root now yield
  \begin{align*}
      \left(\norm{\bu_N}_{\bV}^2 + \norm{p_M}_{\bV}^2\right)^{1/2} &\leq \left[\left( \norm{\bfe}_{\bV^*} +\norm{\bv_N}_{\bV} \right)^2 +\norm{\bfe}_{\bV^*}^2 + \norm{\bv_N}_{\bV}^2 \right. \\
      & \left.\quad +\left( \norm{f_f}_{\bV^*} + \norm{\bv_N}_{\bV} \right)^2\right]^{1/2} \\
      & \leq 2\norm{\bfe}_{\bV^*} +  \norm{f_f}_{\bV^*}  + 3 \norm{\bv_N}_{\bV}
  \end{align*}
  for any $\bv_N \in \bK_N$.
\end{proof}

We remark that a stability estimate with respect to the gap function depends explicitly on the way the constraint in $\bK$ is discretized for $\bK_N$. As that is not specified here we get $3 \inf_{\bv_N \in \bK_N} \norm{\bv_N}_{\bV}$ and an additional $\norm{\bfe}_{\bV^*}$ compared to the a priori estimate of Corollary~\ref{cor:aprioriEst}.

{
In \cite{Banz2022OptimalControl} a combination of the first Strang-Lemma and Falk-Theorem has been proven to obtain an a priori error estimate in the context of an optimal control problem. That idea can be applied here to our reduced problems \eqref{ReducedVI} and \eqref{DiscreteReducedVI} as well.}

\begin{lemma} \label{lem:pre_apriori_error_u}
There holds
\begin{align*}
    \norm{\bu - \bu_N}_{\bV}^2 &\leq 18 \norm{\bu - \bv_N}_{\bV}^2 + 4 \langle D\bu - \bl, \bv_N-\bu + \bv - \bu_N \rangle \\
    & \quad + 4 \norm{(D-D_M)\bv_N - (\bl-\bl_M)}_{\bV^*}^2
\end{align*}
for all $\bv \in \bK $ and all $\bv_N \in \bK_N$.
\end{lemma}
\begin{proof}
{
From Lemma~\ref{lem:DM_coercive}, Lemma~\ref{lem:D_coercive}, the two variational inequalities \eqref{ReducedVI} and \eqref{DiscreteReducedVI}, duality arguments and Young's inequality we obtain
\begin{align*}
    \norm{\bu_N - \bv_N}_{\bV}^2 & \leq \langle D_M (\bu_N - \bv_N) ,\bu_N - \bv_N \rangle \\
    & \leq \langle \bl_M - D_M \bv_N,\bu_N - \bv_N \rangle + \langle D \bu - \bl,\bv -\bu \rangle \\
    & = \langle D\bu - \bl,\bv_N - \bu + \bv - \bu_N \rangle + \langle D (\bv_N-\bu),\bv_N -\bu_N \rangle \\
    & \quad + \langle (D-D_M)\bv_N - (\bl-\bl_M),\bu_N-\bv_N\rangle \\
    & \leq \langle D\bu - \bl,\bv_N - \bu + \bv - \bu_N \rangle + 4\norm{\bu - \bv_N}_{\bV}^2 + \frac{1}{2}\norm{\bu_N - \bv_N}_{\bV}^2 \\
    & \quad  + \norm{(D-D_M)\bv_N - (\bl-\bl_M)}_{\bV^*}^2 .
\end{align*}
Hence,
\begin{align*}
    \norm{\bu_N - \bv_N}_{\bV}^2 & \leq 2\langle D\bu - \bl,\bv_N - \bu + \bv - \bu_N \rangle + 8\norm{\bu - \bv_N}_{\bV}^2  \\
    & \quad  + 2\norm{(D-D_M)\bv_N - (\bl-\bl_M)}_{\bV^*}^2 .
\end{align*}
Now, triangle inequality, namely $\norm{\bu - \bu_N}_{\bV}^2 \leq 2\norm{\bu - \bv_N}_{\bV}^2 + 2\norm{\bu_N - \bv_N}_{\bV}^2$, yields the assertion.}
\end{proof}

For a complete a priori error estimate we need to estimate the discretization error of the left and right hand side, namely $\norm{(D-D_M)\bv_N - (\bl-\bl_M)}_{\bV^*}$, further and {to} extend the error estimate to also bound the error in $p$.
\begin{theorem} \label{thm:StrangFalk}
  There holds
  \begin{align*}
    \norm{\bu - \bu_N}_{\bV}^2 &\leq 42 \norm{\bu - \bv_N}_{\bV}^2 + 12\norm{p-q_M}_{V}^2 \\
    & \quad + 4 \langle \bsigma_\bn(\bu,p) , (\bv_N-\bu + \bv - \bu_N)\cdot \bn \rangle_{\Gamma_c} \\
    \norm{p-p_M}_{V}^2 &  \leq  126 \norm{\bu - \bv_N}_{\bV}^2 + 39 \norm{p-q_M}_{V}^2 \\
        & \quad + 12 \langle \bsigma_\bn(\bu,p) , (\bv_N-\bu + \bv - \bu_N)\cdot \bn \rangle_{\Gamma_c}
\end{align*}
for all $\bv \in \bK $, all $\bv_N \in \bK_N$ and all $q_M \in V_M$.
\end{theorem}
\begin{proof}
Using \eqref{eq:contactPressure2Residuum} we have
\begin{align*}
  \langle D\bu - \bl, \bv_N-\bu + \bv - \bu_N \rangle = \langle \bsigma_\bn(\bu,p) , (\bv_N-\bu + \bv - \bu_N)\cdot \bn \rangle_{\Gamma_c}.
\end{align*}
  Let $G_M = C^{-1}_MC$ be the Galerkin projection operator. With \eqref{SolutionP} we obtain
  \begin{align*}
  -\left(\bl - \bl_M\right) &= B^\top C^{-1} f_f - B^\top C^{-1}_M f_f  \\
  & =  B^\top\left(-p - C^{-1}B \bu\right) - B^\top C^{-1}_M \left(-Cp - B \bu \right)  \\
  &= B^\top \left(G_M-I\right) p + B^\top \left(C^{-1}_M-C^{-1} \right) B \bu.
  \end{align*}
  Hence,
  \begin{align*}
      (D-D_M)\bv_N - (\bl-\bl_M) &= B^\top \left(C^{-1} -C^{-1}_M \right)B \bv_N  - (\bl-\bl_M) \\
      & = B^\top \left(C^{-1} -C^{-1}_M \right)B \left(\bv_N - \bu \right) + B^\top \left(G_M-I\right) p.
  \end{align*}
  The continuity of $B$, $B^\top$, $C^{-1}$ (Lemma~\ref{lem:D_coercive}) and $C^{-1}_M$ (Lemma~\ref{lem:DM_coercive}), and the triangle inequality as well as Young's inequality now yield
  \begin{align*}
      &4\norm{(D-D_M)\bv_N - (\bl-\bl_M)}_{\bV^*}^2 \\
      & \quad \leq 4\norm{ \left(C^{-1} -C^{-1}_M \right)B \left(\bv_N - \bu \right) + \left(G_M-I\right) p}_{V}^2  \\
      & \quad \leq 12 \left(\norm{  C^{-1} B \left(\bv_N - \bu \right) }_{V}^2 + \norm{  C^{-1}_M B \left(\bv_N - \bu \right)  }_{V}^2 + \norm{\left(I-G_M\right) p}_{V}^2 \right) \\
      & \quad \leq  24\norm{ B \left(\bv_N - \bu \right) }_{V^*}^2  + 12 \norm{\left(I-G_M\right) p}_{V}^2   \\
      & \quad \leq 24 \norm{ \bu - \bv_N }_{\bV}^2  + 12 \norm{\left(I-G_M\right) p}_{V}^2.
  \end{align*}
  Since $\left(I-G_M\right) p$ is the Galerkin projection error of $p$ onto $V_M$, C\'ea's lemma implies with $c(\cdot,\cdot)$ inducing the norm $\norm{\cdot}_{V}$ that
  \begin{align*}
      12 \norm{\left(I-G_M\right) p}_{V}^2 = 12 \inf_{q_M \in V_M} \norm{p-q_M}_{V}^2.
  \end{align*}
Now, the first assertion follows with Lemma~\ref{lem:pre_apriori_error_u}.
  
  For the second assertion we note that subtracting \eqref{VIabstractB} from \eqref{DiscreteVIabstractB} yields 
  \begin{align*}
      c(p-p_M,q_M) = - b(\bu-\bu_M,q_M) \quad \forall \ q_M \in V_M.
  \end{align*}
  Hence,
  \begin{align*}
      \norm{p-p_M}_{V}^2 &= c(p-p_M,p-p_M) \\
      &= c(p-p_M,p-q_M) - b(\bu-\bu_M,q_M-p_M) \\
      & \leq \norm{p-p_M}_{V}\norm{p-q_M}_{V} + \norm{\bu-\bu_M}_{\bV}\norm{q_M-p_M}_{V} \\
      & \leq \norm{p-p_M}_{V}\norm{p-q_M}_{V} + \norm{\bu-\bu_M}_{\bV}\left(\norm{p-q_M}_{V} + \norm{p-p_M}_{V} \right) \\
      & \leq \frac{1}{2}\norm{p-p_M}_{V}^2   +\frac{3}{2}\norm{\bu-\bu_M}_{\bV}^2 + \frac{3}{2} \norm{p-q_M}_{V}^2.
  \end{align*}
  Inserting the first assertion into the above estimate yields
  \begin{align*}
        \norm{p-p_M}_{V}^2 & \leq 3 \norm{\bu-\bu_M}_{\bV}^2 + 3 \norm{p-q_M}_{V}^2 \\
        & \leq  126 \norm{\bu - \bv_N}_{\bV}^2 + 39 \norm{p-q_M}_{V}^2 \\
        & \quad + 12 \langle \bsigma_\bn(\bu,p) , (\bv_N-\bu + \bv - \bu_N)\cdot \bn \rangle_{\Gamma_c} .
  \end{align*}
\end{proof}

As usual for contact problems the a priori error estimate consists of some best approximation terms squared and some linear terms measuring among others the non-conformity of the discretiztion $\bK_N \not\subset \bK$. We emphasize that the above error estimate is {quasi-}optimal in $p$.

\section{\texorpdfstring{A $hp$-finite element discretization}{A hp-finite element discretization}} 
\label{sec:hp_discretization}

Let $\mathcal{T}_h$ and $\mathcal{T}_k$ be two independent locally quasi-uniform finite element meshes of $\Omega$ consisting of quadrilaterals or hexahedrons. These meshes induce the set of element edges/faces $\mathcal{E}_h$ and $\mathcal{E}_k$ which are assumed to respect the boundary decomposition into $\Gamma_d$, $\Gamma_t$, $\Gamma_c$, or into $\Gamma_p$, $\Gamma_f$, respectively. 
The diameter of an element $T \in \mathcal{T}_h$ is $h_T$ and its polynomial degree is $r_T \geq 1$. $k_T$ and $s_T \geq 1$ for $T \in \mathcal{T}_k$ is the local element size, polynomial degree, respectively, for the other mesh $\mathcal{T}_k$. The polynomial degree of neighboring elements is assumed to be comparable in the usual sense. We denote by $h$, $k$, $r$ and $s$ the vectors of $(h_T)_{T \in \mathcal{T}_h}$, $(k_T)_{T \in \mathcal{T}_k}$, $(r_T)_{T \in \mathcal{T}_h}$ and $(s_T)_{T \in \mathcal{T}_k}$ and the global mesh size and global polynomial degree
$$ h = \max_{T \in \mathcal{T}_h} h_T, \quad k = \max_{T \in \mathcal{T}_k} k_T, \quad 
r = \min_{T \in \mathcal{T}_h} r_T, \quad 
s = \min_{T \in \mathcal{T}_k} s_T
$$
alike. Its meaning is always clear form the context.
With each element $T\in \mathcal{T}_h$ and $T\in \mathcal{T}_k$ we associate a (bi/tri)-linear bijective mapping $\mathcal{F}_T$ from the reference element $\hat{T} =[-1,1]^d$ onto the physical element $T$.
The standard $hp$-finite element spaces are
\begin{align*}
    V_{ks} &= \left\{ v_{ks} \in V : v_{ks}|_T \circ \mathcal{F}_T \in \mathbb{P}_{s_T}(\hat{T}) \ \forall \ T \in \mathcal{T}_k \right\} \\
    \bV_{hr} &= \left\{ \bv_{hr} \in \bV : \bv_{hr}|_T \circ \mathcal{F}_T \in [\mathbb{P}_{r_T}(\hat{T})]^d \ \forall \ T \in \mathcal{T}_h \right\}
\end{align*}
where $\mathbb{P}_p(\hat{T}) = \operatorname{span}\{x_1^{i} x_2^{j}: 0 \leq i,j \leq p\}$ or $\mathbb{P}_p(\hat{T}) = \operatorname{span}\{x_1^{i} x_2^{j}x_3^{k}: 0 \leq i,j,k \leq p\}$ depending on the dimension $d \in \{2,3\}$.
Every edge/face $e$ in these meshes is a straight line or quadrilateral and thus there exists an affine/bilinear mapping $\mathcal{F}_e:[-1,1]^{d-1} \rightarrow e$ which is just the restriction of $\mathcal{F}_T$ onto the right local edge/face. The set of edges/faces $e$ lying on the contact boundary $\Gamma_c$ is denoted by $\mathcal{E}_{h,\Gamma_c}$. 
Let $r_e$, $s_e$ be the polynomial degree on the edge/face $e \in \mathcal{E}_h$, $e \in \mathcal{E}_k$, respectively, determined by the minimum rule. We emphasize that $r_e = r_T$ for $e \in \mathcal{E}_{h,\Gamma_c}$. We denote by $\{\hat{x}_{i,e}\}_{i=1,\ldots,n_e} \subset [-1,1]^{d-1}$ with $n_e = (r_e+1)^{d-1}$ the (tensor product) Gauss-Lobatto quadrature points.
Therewith,
\begin{align*}
    \bK_{hr} = \left\{ \bv_{hr} \in \bV_{hr}: (\bv_{hr}|_e\cdot \bn) \circ \mathcal{F}_e (\hat{x}_{i,e}) \leq g(\mathcal{F}_e (\hat{x}_{i,e})) \ \forall \; 1 \leq i \leq n_e ,\  e \in \mathcal{E}_{h,\Gamma_c}  \right\}.
\end{align*}
{If $g\in C^0(\Gamma_c)$ then $\bK_{hr}$ is well defined, and a closed, convex subset of $\bV_{hr}$. As $g$ is also sufficiently large, e.g.~$g \geq 0$, by previous assumptions we have that $\bK_{hr}$ is non-empty.}

Let $\mathcal{I}_{ks}:V \rightarrow V_{ks}$ be the $H^1(\Omega)$-projection for which there holds by classical interpolation estimates and real interpolation of Sobolev spaces
\begin{align} \label{eq:interpol_error_est_I}
       \|v - \mathcal{I}_{ks} v \|_{H^1(\Omega)} & \leq C_I \frac{k^{\min\{s,\beta_p-1\}}}{s^{\beta_p-1}} |v|_{H^{\beta_p}(\Omega)} 
\end{align}
for any $\beta_p \geq 1$. We denote by $\mathcal{J}_{hr}:\bH^{\beta_u}(\Omega) \cap \bV \rightarrow \bV_{hr}$ with $\beta_u>(d+1)/2$ the (componentwise applied) tensor product based Gauss-Lobatto nodal interpolation {operator} from \cite{bernardi1992polynomial}. Then there holds by scaling and \cite{bernardi1992polynomial} that
\begin{align} \label{eq:interpol_error_est_J}
     \frac{r}{h}\|\bv - \mathcal{J}_{hr} \bv \|_{\bL^2(\Omega)} + \|\bv - \mathcal{J}_{hr} \bv \|_{\bH^1(\Omega)} & \leq C_J \frac{h^{\min\{r,\beta_u-1\}}}{r^{\beta_u-1}} |\bv|_{\bH^{\beta_u}(\Omega)}.
\end{align}
We emphasize that $\mathcal{J}_{hr}:\bH^{\beta_u}(\Omega) \cap \bK \rightarrow \bK_{hr}$ by the construction of $\bK_{hr}$.
As a final and third interpolation operator we need the (tensor product based) Gauss-Lobatto nodal interpolation $\mathcal{Q}_{hr}: H^{\beta_g}(\Gamma_c) \rightarrow \gamma(\bV_{hr})$ with, see \cite{bernardi1992polynomial},
\begin{align} 
      \|v - \mathcal{Q}_{hr} v \|_{L^2(\Gamma_c)} & \leq C_Q \frac{h^{\min\{r+1,\beta_g\}}}{r^{\beta_g}} |v|_{H^{\beta_g}(\Gamma_c)} \quad \text{for } \beta_g> \frac{d-1}{2}  \label{eq:interpol_error_est_Q_L2}\\
      \|v - \mathcal{Q}_{hr} v \|_{H^{1/2}(\Gamma_c)} & \leq C_Q \frac{h^{\min\{r+1/2,\beta_g-1/2\}}}{r^{\beta_g-1/2}} |v|_{H^{\beta_g}(\Gamma_c)} \quad \text{for } \beta_g> \frac{d-1}{2} + \frac{1}{4}.  \label{eq:interpol_error_est_Q_H12}
\end{align}

With the finite element sets, the interpolation operators and the abstract results from Section \ref{sec:abstractGalerkin} at hand we may analyze the $hp$-finite element Biot contact problem: Find $(u_{hr},p_{ks}) \in \bK_{hr} \times V_{ks}$ such that 
\begin{subequations}\label{hpVIabstract}
\begin{alignat}{2}
a(\bu_{hr},\bv_{hr} - \bu_{hr}) -b(\bv_{hr} - \bu_{hr},p_{ks}) &\geq \langle \bfe, \bv_{hr} - \bu_{hr} \rangle
& \quad& \forall\  \bv_{hr} \in \bK_{hr} \label{hpVIabstractA}\\
-b(\bu_{hr},q_{ks})-c(p_{ks},q_{ks})&=\langle f_f, q_{ks} \rangle &\quad& \forall\ q_{ks} \in V_{ks}  . \label{hpVIabstractB}
\end{alignat}
\end{subequations}

\begin{corollary} \label{cor:stability_fe_sol}
Let $0\leq g \in H^{\beta_g}(\Gamma_c)$ with $\beta_g > (d-1)/2$. Then there exists exactly one solution to \eqref{hpVIabstract}. That solution satisfies the stability estimate
  \begin{align*}
\left(\norm{{
\bu_{hr}}}_{\bV}^2 + \norm{{p_{ks}}}_{\bV}^2 \right)^{1/2}
     &\leq \norm{\bfe}_{\bV^*}+\norm{f_f}_{V^*}  .
  \end{align*}
\end{corollary}

\begin{proof}
From $H^{\beta_g}(\Gamma_c)\subset C^0(\Gamma_c)$ for $\beta_g > (d-1)/2$ we obtain the well posedness of $\bK_{hr}$ and from Theorem~\ref{thm:DiscreteExistence} the unique existence of a discrete FE-solution as well as
\begin{align*}
  \left(\norm{{
  \bu_{hr}}}_{\bV}^2 + \norm{{
  p_{ks}}}_{\bV}^2 \right)^{1/2}
     &\leq \norm{\bfe}_{\bV^*}+\norm{f_f}_{V^*}.
\end{align*}
\end{proof}

The rest of this section is dedicated to the derivation of a priori error estimates. To this end, we use the following abbreviated notation. We simply
write $a(\xi) \lesssim b(\xi)$ if there exists a positive constant $C$ such that $a(\xi) \leq Cb(\xi)$ holds
for all admissible $\xi$. Further, we write $a(\xi) \equiv b(\xi)$ if both $a(\xi) \lesssim b(\xi)$ and $b(\xi) \gtrsim a(\xi)$ hold.
\begin{lemma} \label{lem:conver_rate_non_conformity}
Let $\bv_{hr} \in \bK_{hr}$ and $g \in H^{\frac{d-1}{2} + \frac{1}{4}+\epsilon}(\Gamma_c)$ for an $\epsilon>0$, then
\begin{align*}
    \inf_{\bv \in \bK} \| \gamma(\bv - \bv_{hr}) \|_{L^2(\Gamma_c)} & \lesssim   \left(\frac{h}{r}\right)^{1/2}\left(\|\bv_{hr} \|_{\bH^1(\Omega)} + \| g  \|_{H^{\frac{d-1}{2} + \frac{1}{4}+\epsilon}(\Gamma_c)} \right).
\end{align*}
\end{lemma}

\begin{proof}
  We follow the ideas presented in \cite{maischak2005adaptive,gwinner2013hp} and generalize them to the three dimensional case.  Let $\bv - \bv_\bn \cdot \bn = \bv_{hr} - \gamma(\bv_{hr}) \cdot \bn$ on $\Gamma$ and $\bv_\bn = \min \{\gamma(\bv_{hr}) - \mathcal{Q}_{hr}(g),0 \}+g$ on $\Gamma_c$. As $\overline{\Gamma_c}\cap \overline{\Gamma_d} = \emptyset$ there exists an extension of $\bv_\bn$ onto the whole $\Gamma$ such that $\bv$ now defined on $\Gamma$ can be lifted to some $\bv \in \bK$. We emphasize that $\bv_\bn \leq g$ on $\Gamma_c$ and that $\min \{\gamma(\bv_{hr}) - \mathcal{Q}_{hr}(g),0 \}$ is continuous, piecewise polynomial and thus in $H^{3/2-\epsilon}(\Gamma_c)$ for any $\epsilon>0$. Moreover, on $\Gamma_c$, $\gamma(\bv-\bv_{hr}) = \min \{0, \mathcal{Q}_{hr}(g) - \gamma(\bv_{hr}) \}+g -\mathcal{Q}_{hr}(g) $ and
  \begin{align*}
      \mathcal{Q}_{hr} \left( \min \{0, \mathcal{Q}_{hr}(g) - \gamma(\bv_{hr}) \} \right)=\mathcal{Q}_{hr} \left( \min \{0, g - \gamma(\bv_{hr}) \} \right) = 0  \end{align*}
  since $\bv_{hr} \in \bK_{hr}$. If $d=2$ we obtain from \eqref{eq:interpol_error_est_Q_L2}
  \begin{align*}
      \|\min \{0, \mathcal{Q}_{hr}(g) - \gamma(\bv_{hr}) \} - 0 \|_{L^2(\Gamma_c)} &\leq \| \mathcal{Q}_{hr}(g) - \gamma(\bv_{hr})   \|_{L^2(\Gamma_c)}, \\
      \|\min \{0, \mathcal{Q}_{hr}(g) - \gamma(\bv_{hr}) \} - 0 \|_{L^2(\Gamma_c)} &\leq C_Q \frac{h}{r} \|\min \{0, \mathcal{Q}_{hr}(g) - \gamma(\bv_{hr}) \} \|_{H^1(\Gamma_c)} \\
      &\leq C_Q \frac{h}{r} \| \mathcal{Q}_{hr}(g) - \gamma(\bv_{hr})   \|_{H^1(\Gamma_c)}.
  \end{align*}
  Real interpolation of Sobolev spaces, see e.g.~\cite[Prob.~14.1.5, Sec.~14.3]{brenner2008mathematical}, yields
  \begin{align*}
      \|\min \{0, \mathcal{Q}_{hr}(g) - \gamma(\bv_{hr}) \} - 0 \|_{L^2(\Gamma_c)} \leq C \left(\frac{h}{r}\right)^{1/2}\| \mathcal{Q}_{hr}(g) - \gamma(\bv_{hr})   \|_{H^{1/2}(\Gamma_c)}.
  \end{align*}
If $d=3$, we need $\beta_g>1$ in \eqref{eq:interpol_error_est_Q_L2}, i.e.  
   \begin{align*}
      \|\min \{0, \mathcal{Q}_{hr}(g) - \gamma(\bv_{hr}) \} - 0 \|_{L^2(\Gamma_c)} &\leq \| \min \{0, \mathcal{Q}_{hr}(g) - \gamma(\bv_{hr}) \}  \|_{L^2(\Gamma_c)} \\
      \|\min \{0, \mathcal{Q}_{hr}(g) - \gamma(\bv_{hr}) \} - 0 \|_{L^2(\Gamma_c)} &\leq C_Q \left(\frac{h}{r}\right)^{1+\epsilon}   \|\min \{0, \mathcal{Q}_{hr}(g) - \gamma(\bv_{hr}) \} \|_{H^{1+\epsilon}(\Gamma_c)}
  \end{align*}
  for some arbitrarily small $\epsilon>0$. As before we apply real interpolation of Sobolev spaces but noting that for Hilbert spaces there even holds, see \cite[Eq.~14.2.5]{brenner2008mathematical},
  \begin{align}
      \left[ L^2(\Gamma_c), H^{1+\epsilon}(\Gamma_c) \right]_{\frac{1}{2(1+\epsilon)},2} = H^{1/2}(\Gamma_c)
  \end{align}
  with equivalent norms.
  Hence,
  \begin{align*}
      \|\min \{0, \mathcal{Q}_{hr}(g) - \gamma(\bv_{hr}) \} - 0 \|_{L^2(\Gamma_c)} &\lesssim \left(\frac{h}{r}\right)^{1/2}\|\min \{0, \mathcal{Q}_{hr}(g) - \gamma(\bv_{hr}) \}   \|_{H^{1/2}(\Gamma_c)} \\
      & \lesssim \left(\frac{h}{r}\right)^{1/2}\|\mathcal{Q}_{hr}(g) - \gamma(\bv_{hr})   \|_{H^{1/2}(\Gamma_c)}.
  \end{align*}  
  Thus,
  \begin{align*}
      \| \gamma(\bv - \bv_{hr}) \|_{L^2(\Gamma_c)} & \leq \|\min \{0, \mathcal{Q}_{hr}(g) - \gamma(\bv_{hr}) \} - 0 \|_{L^2(\Gamma_c)} + \|g - \mathcal{Q}_{hr}(g)\|_{L^2(\Gamma_c)} \\
      & \lesssim  \left(\frac{h}{r}\right)^{1/2}\| \mathcal{Q}_{hr}(g) - \gamma(\bv_{hr})   \|_{H^{1/2}(\Gamma_c)} + \|g - \mathcal{Q}_{hr}(g)\|_{L^2(\Gamma_c)}\\
      & \lesssim  \left(\frac{h}{r}\right)^{1/2}\left(\|\gamma(\bv_{hr})   \|_{H^{1/2}(\Gamma_c)} + \| g  \|_{H^{1/2}(\Gamma_c)}  \right) \\
      & \quad +\left(\frac{h}{r}\right)^{1/2} \| g- \mathcal{Q}_{hr}(g)  \|_{H^{1/2}(\Gamma_c)}+  \|g - \mathcal{Q}_{hr}(g)\|_{L^2(\Gamma_c)}\\
       & \lesssim  \left(\frac{h}{r}\right)^{1/2}\left(\|\gamma(\bv_{hr})  \|_{H^{1/2}(\Gamma_c)} + \| g  \|_{H^{\frac{d-1}{2} + \frac{1}{4}+\epsilon}(\Gamma_c)} \right)
  \end{align*}
  by \eqref{eq:interpol_error_est_Q_H12} with an arbitrary small $\epsilon>0$. Now, the assertion follows with the trace estimate for $\|\gamma(\bv_{hr})   \|_{H^{1/2}(\Gamma_c)}$.
\end{proof}

The guaranteed convergence rate of the non-conformity error due to $\bK_{hr} \not\subset \bK$ can be improved for the lowest order $h$-version.

\begin{corollary} \label{cor:conver_rate_non_conformity_linear}
Let $r=1$, $\bv_{hr} \in \bK_{hr}$ and $g \in H^{\beta_g}(\Gamma_c)$ with $\beta_g >\frac{d-1}{2} $, then there holds
\begin{align*}
    \inf_{\bv \in \bK} \| \gamma(\bv - \bv_{hr}) \|_{L^2(\Gamma_c)} & \lesssim    h^{\min\{2,\beta_g\}}  \| g  \|_{H^{\beta_g}(\Gamma_c)} .
\end{align*}
\end{corollary}
\begin{proof}
  The proof follows the lines of Lemma~\ref{lem:conver_rate_non_conformity}, but we note that for $r=1$ the interpolant $\mathcal{Q}_{hr}g$ is continuous, piecewise linear just like $\gamma(\bv_{hr})$. As the non-penetration condition in $\bK_{hr}$ is enforced in the vertices of the mesh we have $\gamma(\bv_{hr}) \leq \mathcal{Q}_{hr}g $ everywhere on $\Gamma_c$. Hence, the $\gamma(\bv - \bv_{hr} ) = g - \mathcal{Q}_{hr}g$ for the $\bv \in \bK$ constructed in the proof of Lemma~\ref{lem:conver_rate_non_conformity}. The interpolation error estimate \eqref{eq:interpol_error_est_Q_L2} yields the assertion.
 \end{proof}

With the previous results at hand we can easily prove convergence of the finite element method.

\begin{theorem}
If $0 \leq g \in H^{\frac{d-1}{2} + \frac{1}{4}+\epsilon}(\Gamma_c)$ for an $\epsilon>0$ then there holds
  \begin{align*}
      \lim_{h/r, k/s \rightarrow 0} \norm{\bu - \bu_{hr}}_{\bV}^2 +\norm{p-p_{ks}}_{V}^2 = 0.
  \end{align*}
\end{theorem}

\begin{proof}
We apply the general a priori error estimate of Theorem~\ref{thm:StrangFalk} and show that the right hand side tends to zero.

Since $\mathcal{J}_{hr}:\bH^2(\Omega) \cap \bK \rightarrow \bK_{hr}$ and $\bH^2(\Omega) \cap \bK$ is dense in $\bK$ there holds by classical density arguments that
\begin{gather*}
    \lim_{h/r \rightarrow 0} \inf_{\bv_{hr} \in \bK_{hr}}  \| \bu - \bv_{hr}\|_{\bH^1(\Omega)}=0 \text{ and }
    \lim_{k/s \rightarrow 0} \inf_{q_{ks} \in V_{ks}} \| p - q_{ks}\|_{H^1(\Omega)}=0 .
\end{gather*}
Duality and the trace theorem imply
\begin{align*}
    \langle \bsigma_\bn(\bu,p) , \gamma(\bv_{hr}-\bu ) \rangle_{\Gamma_c} \leq \|\bsigma_\bn(\bu,p) \|_{\tilde{H}^{-1/2}(\Gamma_c)} \|\bv_{hr}-\bu\|_{\bH^1(\Omega)}.
\end{align*}
With the equivalence of the norms, Lemma~\ref{lem:norm_equi}, it remains to estimate $\langle \bsigma_\bn(\bu,p) , \gamma(\bv - \bu_{hr}) \rangle_{\Gamma_c}$. As $L^2(\Gamma_c)$ is dense in $\tilde{H}^{-1/2}(\Gamma_c)$, we know that for any $\epsilon>0$ there exists a $\bsigma_\epsilon \in L^2(\Gamma_c)$ with $\|\bsigma_\epsilon-\bsigma_\bn(\bu,p) \|_{\tilde{H}^{-1/2}(\Gamma_c)} < \epsilon$. Hence, with Lemma~\ref{lem:conver_rate_non_conformity} including the $\bv$ constructed in its proof, trace estimate, \eqref{eq:interpol_error_est_Q_H12} and Corollary~\ref{cor:stability_fe_sol} we obtain
\begin{align*}
    \langle \bsigma_\bn(\bu,p) , \gamma(\bv - \bu_{hr}) \rangle_{\Gamma_c} &= \langle \bsigma_\bn(\bu,p)-\bsigma_\epsilon , \gamma(\bv - \bu_{hr}) \rangle_{\Gamma_c} + \langle \bsigma_\epsilon , \gamma(\bv - \bu_{hr}) \rangle_{\Gamma_c} \\
    & \leq \|\bsigma_\epsilon-\bsigma_\bn(\bu,p) \|_{\tilde{H}^{-1/2}(\Gamma_c)} \|\gamma(\bv - \bu_{hr}) \|_{H^{1/2}(\Gamma_c)}  \\
    & \quad + \|\bsigma_\epsilon \|_{L^2(\Gamma_c)}  \| \gamma(\bv - \bu_{hr})\|_{L^2(\Gamma_c)} \\
    & \lesssim \epsilon \left( \|\gamma(\bu_{hr})\|_{H^{1/2}(\Gamma_c)} + \|g - \mathcal{Q}_{hr} g\|_{H^{1/2}(\Gamma_c)} + \|g\|_{H^{1/2}(\Gamma_c)} \right) \\
    & \quad + \|\bsigma_\epsilon \|_{L^2(\Gamma_c)} \left(\frac{h}{r}\right)^{1/2}\left(\|\bv_{hr} \|_{\bH^1(\Omega)} + \| g  \|_{H^{\frac{d-1}{2} + \frac{1}{4}+\epsilon}(\Gamma_c)} \right) \\
    & \lesssim \epsilon \left( \|\bu_{hr} \|_{\bH^1(\Omega)} + \| g  \|_{H^{\frac{d-1}{2} + \frac{1}{4}+\epsilon}(\Gamma_c)} \right) \\
    & \quad + \|\bsigma_\epsilon \|_{L^2(\Gamma_c)} \left(\frac{h}{r}\right)^{1/2}\left(\|\bu_{hr} \|_{\bH^1(\Omega)} + \| g  \|_{H^{\frac{d-1}{2} + \frac{1}{4}+\epsilon}(\Gamma_c)} \right) \\
    & \lesssim \epsilon + \|\bsigma_\epsilon \|_{L^2(\Gamma_c)} \left(\frac{h}{r}\right)^{1/2}.
\end{align*}
Choosing $h/r$ sufficiently small yields $\langle \bsigma_\bn(\bu,p) , \gamma(\bv - \bu_{hr}) \rangle_{\Gamma_c} \lesssim \epsilon$ and thus the assertion.
\end{proof}

It is possible to reduce the regularity requirement for the convergence theorem from $g \in H^{\frac{d-1}{2} + \frac{1}{4}+\epsilon}(\Gamma_c)$ to $g \in H^{\frac{d-1}{2}+\epsilon}(\Gamma_c)$, the minimum needed for the discrete problem to be well posed, but it requires a different technique, see \cite{gwinner2013hp} for a two dimensional contact case.

\begin{theorem} \label{thm:convergence_rates}
Let $\bu \in \bH^{\beta_u}(\Omega)$ with $\beta_u>(d+1)/2$, $p \in H^{\beta_p}(\Omega)$ with $\beta_p \geq 1$, $\bsigma_\bn(\bu,p)|_{\Gamma_c} \in L^2(\Gamma_c)$ and $g \in H^{\frac{d-1}{2} + \frac{1}{4}+\epsilon}(\Gamma_c)$ for an $\epsilon>0$. Then there holds
\begin{align*}
    \norm{\bu - \bu_{hr}}_{\bV}^2 +\norm{p-p_{ks}}_{V}^2 \lesssim \left(\frac{h}{r}\right)^{1/2} +   \frac{k^{2\min\{s,\beta_p-1\}}}{s^{2(\beta_p-1)}} .
\end{align*}
\end{theorem}

\begin{proof}
  We apply the general a priori error estimate of Theorem~\ref{thm:StrangFalk} and show that the right hand side tends to zero with a given rate.
  
With Theorem~\ref{thm:StrangFalk}, the equivalence of the norms, see Lemma~\ref{lem:norm_equi}, the Cauchy-Schwarz inequality, $\bsigma_\bn(\bu,p)|_{\Gamma_c} \in L^2(\Gamma_c)$ and a trace estimate there holds
  \begin{align*}
    \norm{\bu - \bu_{hr}}_{\bV}^2 +\norm{p-p_{ks}}_{V}^2 & \lesssim \| \bu - \bv_{hr}\|_{\bH^1(\Omega)}^2 + \|p-q_{ks}\|_{H^1(\Omega)}^2 \\
    & \quad +  \| \bsigma_\bn(\bu,p) \|_{L^2(\Gamma_c)} \| \gamma(\bv_{hr}-\bu + \bv - \bu_{hr})  \|_{L^2(\Gamma_c)} \\
    & \lesssim \| \bu - \bv_{hr}\|_{\bH^1(\Omega)}^2 + \|p-q_{ks}\|_{H^1(\Omega)}^2  \\
    & \quad +  \| \bu - \bv_{hr}\|_{\bL^2(\Omega)}^{1/2} \| \bu - \bv_{hr}\|_{\bH^1(\Omega)}^{1/2} \\
    & \qquad +  \| \gamma(\bv - \bu_{hr}) \|_{L^2(\Gamma_c)}.
\end{align*}
Lemma~\ref{lem:conver_rate_non_conformity}, Corollary~\ref{cor:stability_fe_sol}, $q_{ks} = \mathcal{I}_{ks}p$, \eqref{eq:interpol_error_est_I}, $\bv_{hr} = \mathcal{J}_{hr}\bu$ and \eqref{eq:interpol_error_est_J} imply
 \begin{align*}
    \norm{\bu - \bu_{hr}}_{\bV}^2 +\norm{p-p_{ks}}_{V}^2 & \lesssim \| \bu - \bv_{hr}\|_{\bH^1(\Omega)}^2 + \|p-q_{ks}\|_{H^1(\Omega)}^2  \\
    & \quad +  \| \bu - \bv_{hr}\|_{\bL^2(\Omega)}^{1/2} \| \bu - \bv_{hr}\|_{\bH^1(\Omega)}^{1/2} \\
    & \qquad +  \| \gamma(\bv - \bu_{hr}) \|_{L^2(\Gamma_c)} \\
    & \lesssim \frac{h^{2\min\{r,\beta_u-1\}}}{r^{2(\beta_u-1)}} |\bu|_{\bH^{\beta_u}(\Omega)}^2 +   \frac{k^{2\min\{s,\beta_p-1\}}}{s^{2(\beta_p-1)}} |p|_{H^{\beta_p}(\Omega)}^2  \\
    & \quad +  \frac{h^{\min\{r+1/2,\beta_u-1/2\}}}{r^{\beta_u-1/2}} |\bu|_{\bH^{\beta_u}(\Omega)} \\
    & \qquad + \left(\frac{h}{r}\right)^{1/2}\left(\|\bu_{hr} \|_{\bH^1(\Omega)} + \| g  \|_{H^{\frac{d-1}{2} + \frac{1}{4}+\epsilon}(\Gamma_c)} \right) \\
    & \lesssim \left(\frac{h}{r}\right)^{1/2} +   \frac{k^{2\min\{s,\beta_p-1\}}}{s^{2(\beta_p-1)}} .
\end{align*}
In the last line we eliminate dominated convergence rates knowing that $\beta_u > (d+1)/2 \geq 3/2$.
\end{proof}

We emphasize that the convergence rates in $k$ and $s$ are optimal and are reduced in $h$ and $r$ as can be expected for contact problems \cite{gwinner2013hp}.
In case of the lowest order $h$-version, i.e.~$r=1$, we recover optimal guaranteed convergence rates if everything is  sufficiently smooth.
\begin{corollary} \label{cor:lowest_h_conver_rates}
Let $r=1$, $\bu \in \bH^{\beta_u}(\Omega)$ with $\beta_u>(d+1)/2$, $p \in H^{\beta_p}(\Omega)$ with $\beta_p \geq 1$, $\bsigma_\bn(\bu,p)|_{\Gamma_c} \in L^2(\Gamma_c)$ and $g \in H^{\beta_g}(\Gamma_c)$ with $\beta_g > (d-1)/2 $. Then there holds
\begin{align*}
    \norm{\bu - \bu_{hr}}_{\bV}^2 +\norm{p-p_{ks}}_{V}^2 \lesssim  h^{\min\{2,\beta_u-1/2,\beta_g\}} + \frac{k^{2\min\{s,\beta_p-1\}}}{s^{2(\beta_p-1)}} .
\end{align*}
\end{corollary}
\begin{proof}
  We argue similarly as for Theorem~\ref{thm:convergence_rates} but use Corollary~\ref{cor:conver_rate_non_conformity_linear} instead of Lemma~\ref{lem:conver_rate_non_conformity} and $\gamma(\mathcal{J}_{hr}(\bu)) = \mathcal{Q}_{hr}(\gamma(\bu))$ followed by a trace estimate to obtain
 \begin{align*}
    \norm{\bu - \bu_{hr}}_{\bV}^2 +\norm{p-p_{ks}}_{V}^2 & \lesssim \| \bu - \bv_{hr}\|_{\bH^1(\Omega)}^2 + \|p-q_{ks}\|_{H^1(\Omega)}^2  \\
    & \quad +  \| \gamma(\bu - \bv_{hr})\|_{L^2(\Gamma_c)}  +  \| \gamma(\bv - \bu_{hr}) \|_{L^2(\Gamma_c)} \\
    & \lesssim h^{2\min\{1,\beta_u-1\}} |\bu|_{\bH^{\beta_u}(\Omega)}^2 +   \frac{k^{2\min\{s,\beta_p-1\}}}{s^{2(\beta_p-1)}} |p|_{H^{\beta_p}(\Omega)}^2  \\
    & \quad +  h^{\min\{2,\beta_u-1/2\}} |\gamma(\bu)|_{H^{\beta_u-1/2}(\Gamma_c)} \\
    & \qquad + h^{\min\{2,\beta_g\}} \| g  \|_{H^{\beta_g}(\Gamma_c)}  \\
    & \lesssim h^{\min\{2,\beta_u-1/2,\beta_g\}} + \frac{k^{2\min\{s,\beta_p-1\}}}{s^{2(\beta_p-1)}} .
\end{align*}
\end{proof}

The key ingredient for the proof of convergence rates is the existence of an interpolation operator $\mathcal{J}$ mapping sufficiently smooth functions in $\bK$ onto $\bK_{hr}$ with the additional property that $\gamma(\mathcal{J}\bu) = \mathcal{Q}(\gamma(\bu)) $ for a well analyzed interpolation operator $\mathcal{Q}$. On quadrilaterals and hexahedrons we may exploit the tensor structure to define $\mathcal{J}$ and $\mathcal{Q}$. On e.g.~triangles a general $hp$-result is still open to the best of our knowledge and we would choose $\mathcal{J}$ to be the nodal interpolation in the Fekete points and use the in \cite{taylor2000algorithm} numerically verified conjecture that the Fekete points restricted to the edges of a reference triangle correspond to the one-dimensional Gauss-Lobatto points. Hence, we could then choose $\mathcal{Q}$ as done above.
 
\section{Numerical Experiments}
\label{sec:numerics}
For the numerical experiments we {choose} $\Omega =(0,1)^2$, $\Gamma_d = \Gamma_p = [0,1]\times \{1\}$, $\Gamma_c = [0,1] \times \{0\}$, $\Gamma_t = \partial \Omega \setminus ( \Gamma_d \cup \Gamma_c)$ and $\Gamma_f = \partial \Omega \setminus \Gamma_p$. The data are $\bfe = (0,-1)^\top$, $f_f = -1$ and $g = 3(1-\cos(x_1-0.5))$. The material parameters are $\iota = \tau = \alpha = 1$ and $\bkappa = \bI$. As $\Gamma_d = \Gamma_p$ we use the same FE-mesh and polynomial degree for $\bu$ and $p$, i.e.~$\bV_{hr} = [V_{ks}]^2$ or short $k=h$ and $s=r$. The maximum number of hanging nodes per edge is one and the polynomial degree of adjacent elements does not differ by more than one. 
{In case of adaptivity we use the sum of the local error indicators
\begin{align*}
     \eta_{u,T}^2 &= \frac{h_T^2}{r_T^2} \left\|
    \operatorname{div}\btheta(\bu_{hr}) - \alpha \nabla p_{hr} + \bfe \right\|_{L^2(T)}^2  + \sum_{e \in \partial T \cap \Omega} \frac{h_e}{2r_e} \left\| \llbracket \btheta(\bu_{hr})\cdot \bn \rrbracket \right\|_{L^2(e)}^2 \\
    & \qquad + \sum_{e \in \partial T \cap \Gamma_t} \frac{h_e}{r_e} \left\| \btheta(\bu_{hr})\cdot \bn- \alpha  p_{hr}\bn \right\|_{L^2(e)}^2 \\
   & \quad \qquad + \sum_{e \in \partial T \cap \Gamma_c} \frac{h_e}{r_e} \left\|  \btheta(\bu_{hr})\cdot \bn- \left(\alpha  p_{hr} - \lambda_{hr} \right)\bn  \right\|_{L^2(e)}^2,\\
     \eta_{p,T}^2 &= \frac{h_T^2}{r_T^2} \left\|
    \operatorname{div}\left(\bkappa \nabla p_{hr}\right) -\frac{\alpha^{2}}{\iota} p_{hr}-  \alpha \operatorname{div} \bu_{hr} - f_f  \right\|_{L^2(T)}^2 \\
    & \quad + \sum_{e \in \partial T \cap \Omega} \frac{h_e}{2r_e} \left\| \llbracket \bkappa \nabla p_{hr} \cdot \bn \rrbracket \right\|_{L^2(e)}^2 + \sum_{e \in \partial T \cap \Gamma_f} \frac{h_e}{r_e} \left\| \bkappa \nabla p_{hr} \cdot \bn \right\|_{L^2(e)}^2,
\end{align*}
where $\llbracket \cdot \rrbracket$ is the usual jump, and of
    \begin{align*}
      \left\| \frac{h^{1/2}}{r^{-1/2}} (\lambda_{hr}-\mu) \right\|^2_{L^2(\Gamma_c)} + \left\| \frac{r^{1/2}}{h^{-1/2}} \zeta \right\|^2_{L^2(\Gamma_c)} + \frac{90}{469} \left( \langle\lambda_{hr},\zeta\rangle_{\Gamma_c} + \langle \mu,g-\bu_{hr} \cdot \bn \rangle_{\Gamma_c} \right).
    \end{align*}
    Here, the negative discrete contact stress $\lambda_{hr} \in V_{hr}|_{\Gamma_c} = \gamma(\bV_{hr}) \subset \widetilde{H}^{-1/2}(\Gamma_c)$ is reconstructed by solving
\begin{align} \label{eq:lambda_reconstruction}
  \langle \lambda_{hr}, \bv_{hr}\cdot n \rangle_{\Gamma_c} = -a(\bu_{hr},\bv_{hr} )+b(\bv_{hr} ,p_{ks}) + \langle \bfe ,\bv_{hr} \rangle \quad \forall \  \bv_{hr} \in \bV_{hr},
\end{align}
and the cut-off functions are chosen according to 
    \begin{align*}
       \mu = \sup\left\{0, \lambda_{hr} - \frac{45}{469} \frac{r}{h}(g-\widetilde{\bu}_\bn) \right\}, \quad 
       \zeta = \sup\left\{\bu_{hr} \cdot \bn - g , -\frac{45}{469} \frac{h}{r}\lambda_{hr}\right\}.
    \end{align*}
 The reliability and the efficiency of this error indicator will be analyzed in a further contribution.   
}

 As the opening angles in the corners of $\Omega$ are $\pi/2$ and we have a transition from Dirichlet to Neumann boundary condition in the top left and {top} right corners, we can only expect $\bu \in \bH^2(\Omega)$, provided that the data $(\bfe,\nabla p,{\bsigma_\bn}) \in \bL^2(\Omega) \times \bL^2(\Omega) \times H^{1/2}(\Gamma_c)$ or better. Obviously, $(\bfe,\nabla p) \in \bL^2(\Omega) \times \bL^2(\Omega)$ and from the picture of the (very fine discrete) solution, see Figure~\ref{fig:solution}, we can expect ${\bsigma_\bn} \in H^{3/2-\epsilon}(\Gamma_c)$ for any $\epsilon>0$. For the same reason we can expect $p \in H^2(\Omega)$ provided that the data $(f_f,\div \bu) \in L^2(\Omega) \times L^2(\Omega)$ or better, which is obviously the case here. Moreover, the Dirichlet to Neumann transitions seem to be stronger singularities than the contact to non-contact transition ones.

 \begin{figure}[tbp]
  \centering 
  \subfloat[Deformed body (color stands for value of $p$)\\ and obstacle (red curve)]{
	\includegraphics[trim = 15mm 5mm 15mm 5mm, clip,height=50mm, keepaspectratio]{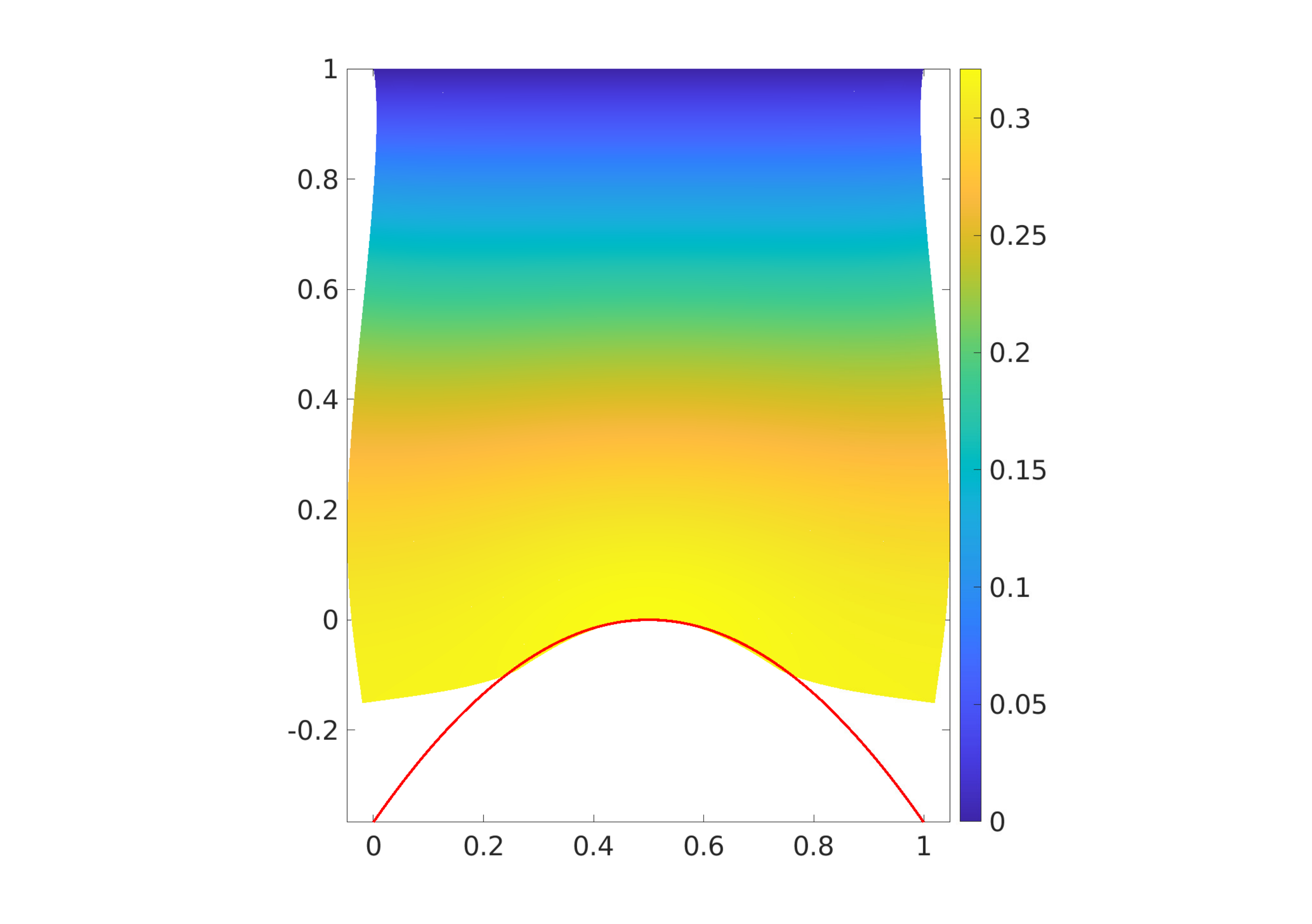}  }
 \subfloat[Negative contact pressure{:  $-\bsigma_\bn(\bu,p)|_{\Gamma_c}$}]{
	\includegraphics[trim = 4mm 5mm 9mm 5mm, clip,height=40mm, keepaspectratio]{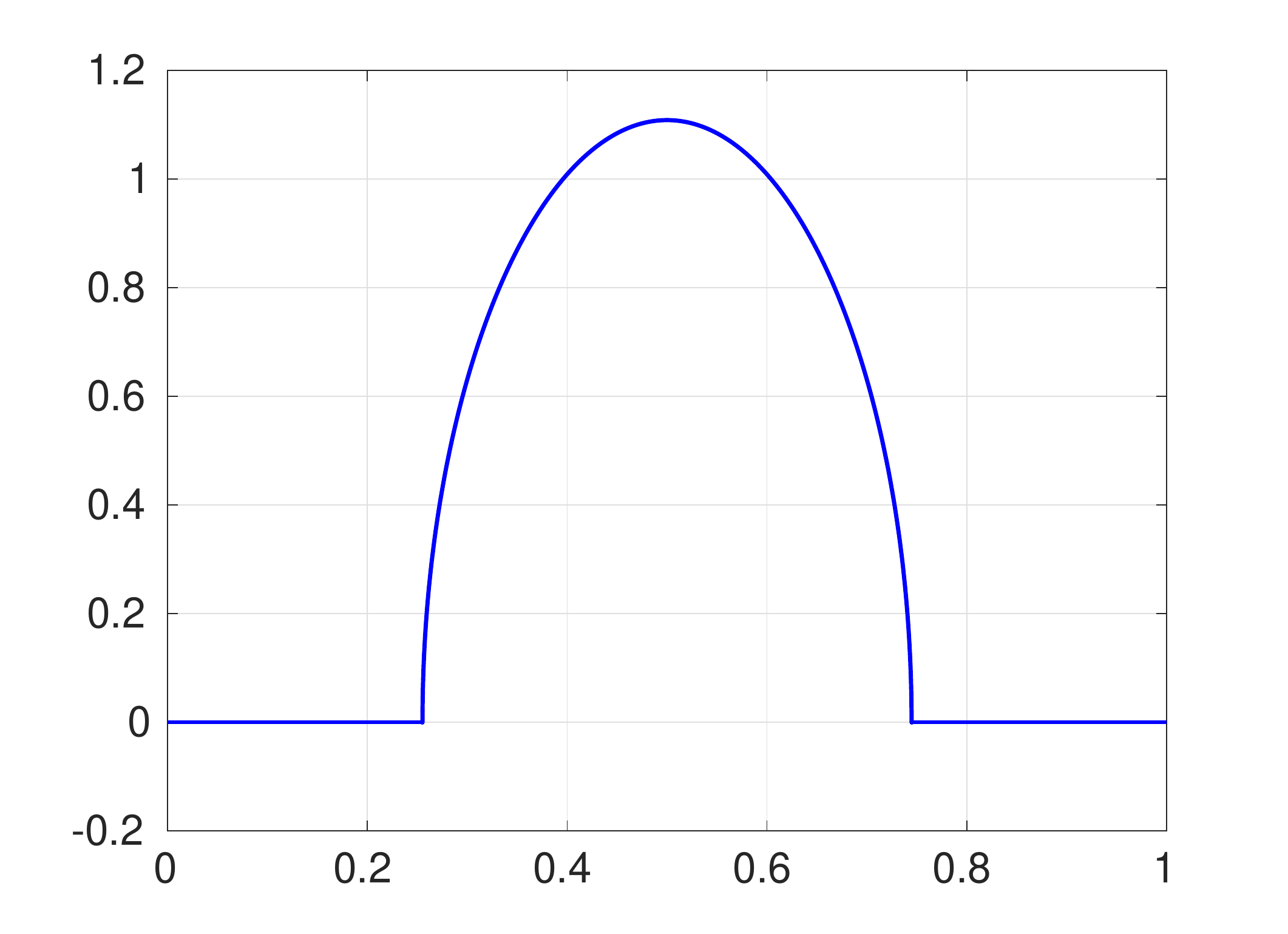}  }
  \caption{Visualization of the solution}
  \label{fig:solution}
\end{figure}

{As no exact solution is known we approximate it by an overrefined Galerkin solution which is obtained by quartering every element and increasing the polynomial degree by one compared to the finest solution of a given particular refinement scheme to compute the error}. The results for different refinement techniques, uniform $h$-version with $r=1,2,3$, {uniform $r$-version with $h = 0.25$}, $h$-adaptive schemes with $r=2,3$ and a $hr$-adaptive scheme are plotted in Figure~\ref{fig:error}. For the adaptive schemes we use { isotropic refinements plus} Dörfer-marking with bulk parameter $1/2$, and decided between $h$ and $r$ refinement based on the decay rate of the error estimator in $r$ to estimate the local Sobolev regularity, see e.g.~\cite{banz2018higher,banz2019higher,banz2020posteriori}. We observe optimal experimental order of convergence (eoc) for the uniform $h$-version with $r=1$, namely 0.5 w.r.t.~the total degrees of freedom $N$ as { predicted by Corollary~\ref{cor:lowest_h_conver_rates}. Contrary to Theorem~\ref{thm:convergence_rates} but common for contact problems we do not observe reduced convergence rates compared to the best approximation of $(\bu,p)$ within higher order schemes}. Increasing $r$ to two or three does not increase the eoc but only improves the error constant indicating $(\bu,p) \in \bH^2(\Omega) \times H^2(\Omega)$ only. {The eoc of the uniform $r$-version increases with $r$ and seems to tend towards one, which is in agreement with the expectation $\bsigma_\bn \in H^{3/2-\epsilon}(\Gamma_c)$ and that the $H^2$ limiting singularities lie in the corners of the meshes.} Pursuing an adaptive strategy we recover optimal order of convergence, that is 1, 1.5 for $h$-adaptivity with $r=2,3$, respectively, and exponential in case of $hr$-adaptivity. The latter can be observed from the straight orange line in the right diagram of Figure~\ref{fig:error}.\\

{
In two dimensions all these singularities are point singularities and, thus, isotropic refinements can lead to exponential convergence. In three dimensions, however, the singularity coinciding with the free boundary typically forms a one dimensional manifold which cannot be adequately resolved by isotropic refinements, and thus, we would expect an algebraic convergence rate of the $hr$-method similarly as for two dimensional obstacle problems \cite{banz2014posteriori,banz2015biorthogonal,banz2019hybridization}.
}

\begin{figure}[tb]
    \centering
    \subfloat{
	\begin{tikzpicture}[scale=0.76]
            \begin{loglogaxis}[
                width=0.65\textwidth,
                line width=1.0pt,
                mark size=3pt,
                xmin=1,xmax=1e7,
                ymin=1e-10,ymax=1e0,
                legend style={font=\large,at={(0,0)},anchor=south west},
                xlabel=$N$
                ]
                \addplot+[mark=+, color=blue]  table[x index=0,y index=2] {h1.txt};
                \addplot+[mark=square, color=red]   table[x index=0,y index=2] {h2.txt};
                \addplot+[mark=triangle, color=teal] table[x index=0,y index=2] {h3.txt};
                \addplot+[mark=10-pointed star, color=black]    table[x index=0,y index=2] {p4.txt};

                \addplot+[mark=x, color=purple]   table[x index=0,y index=2] {a2.txt};
                \addplot+[mark=o, color=darkbrown,solid]    table[x index=0,y index=2] {a3.txt};
                \addplot+[mark=diamond, color=orange,solid] table[x index=0,y index=2] {hp.txt};
 
                \legend{{$h$-unif., $r=1$}, {$h$-unif., $r=2$}, {$h$-unif., $r=3$},{$r$-unif., $h=\frac{1}{4}$}, {$h$-adap., $r=2$}, {$h$-adap., $r=3$}, {$hr$-adap.}}

			\draw (1e5,7e-3) -- (1e6,7e-3);
			\draw (1e6,7e-3) -- (1e6,7e-3*0.316227766016838);
			\draw (1e5,7e-3) -- (1e6,7e-3*0.316227766016838);
			\node at (1.5e6,4e-3) {$\frac{1}{2}$};
			
			\draw (1e5,1e-4) -- (1e6,1e-4);
			\draw (1e6,1e-4) -- (1e6,1e-4*0.1);
			\draw (1e5,1e-4) -- (1e6,1e-4*0.1);
			\node at (1.5e6,4e-5) {$1$};
			
			\draw (1e5,1e-6*0.031622776601684) -- (1e6,1e-6*0.031622776601684);
 			\draw (1e5,1e-6) -- (1e6,1e-6*0.031622776601684);
 			\draw (1e5,1e-6) -- (1e5,1e-6*0.031622776601684);
  			\node at (7e4,2e-7) {$\frac{3}{2}$};
  			
            \end{loglogaxis}
        \end{tikzpicture}
        }
	\subfloat{
	\begin{tikzpicture}[scale=0.76]
            \begin{semilogyaxis}[
                width=0.65\textwidth,
                line width=1.0pt,
                mark size=3pt,
                xmin=0,xmax=200,
                ymin=1e-10,ymax=1e0,
                legend style={font=\large,at={(0,0)},anchor=south west},
                xlabel=$N^{1/3}$
                ]
                \addplot+[mark=+, color=blue]  table[x index=1,y index=2] {h1.txt};
                \addplot+[mark=square, color=red]   table[x index=1,y index=2] {h2.txt};
                \addplot+[mark=triangle, color=teal] table[x index=1,y index=2] {h3.txt};
                 \addplot+[mark=10-pointed star, color=black]    table[x index=1,y index=2] {p4.txt};
                 
                \addplot+[mark=x, color=purple]   table[x index=1,y index=2] {a2.txt};
                \addplot+[mark=o, color=darkbrown]    table[x index=1,y index=2] {a3.txt};
                \addplot+[mark=diamond, color=orange,solid] table[x index=1,y index=2] {hp.txt};


            \end{semilogyaxis}
        \end{tikzpicture}
        }
  \caption{Error vs.~degrees of freedom $N$}
  \label{fig:error}        
\end{figure}
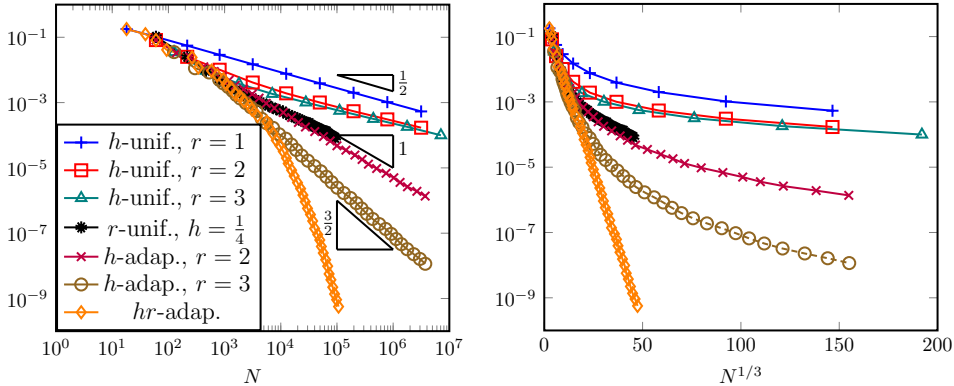



\bibliographystyle{siamplain}
\bibliography{references}
\end{document}